\documentclass{article}

\usepackage{amsmath}
\usepackage{graphicx}
\usepackage{amssymb}
\usepackage{amsfonts}
\usepackage{latexsym}
\usepackage{amscd}
\usepackage[mathscr]{euscript}
\usepackage{xy} \xyoption{all}
\usepackage[utf8]{inputenc} 

\usepackage{tikz}
\usetikzlibrary{shapes,arrows}
\usepackage[utf8]{inputenc}
\usepackage{color}

\usepackage{setspace}

\newtheorem{theorem}{Theorem}[section]

\newtheorem{corollary}[theorem]{Corollary}

\newtheorem{definition}[theorem]{Definition}
\newtheorem{example}[theorem]{Example}

\newtheorem{lemma}[theorem]{Lemma}
\newtheorem{fact}[theorem]{Fact}

\newtheorem{proposition}[theorem]{Proposition}
\newtheorem{remark}[theorem]{Remark}

\newenvironment{proof}[1][Proof]{\noindent\textbf{#1.} }{\ \rule{0.5em}{0.5em}}

\begin{document}

\title{Representations of Leavitt Path Algebras}
\author{Ayten Ko\c{c} $^{*}\> $, $\>$  Murad  \"{O}zaydın$^{**}$\\ \\}




\maketitle

\begin{abstract}
We study representations of a Leavitt path algebra $L$ of a finitely separated digraph $\Gamma$ over a field. We show that the category of $L$-modules  is equivalent to a full subcategory of quiver representations. When $\Gamma$ is a (non-separated) row-finite digraph we determine all possible finite dimensional quotients of $L$ after giving a necessary and sufficient graph theoretic criterion for the existence of a nonzero finite dimensional quotient. This criterion  is also equivalent to $L$ having  UGN (Unbounded Generating Number) as well as being algebraically amenable. 
We also realize the category of $L$-modules as a retract, hence a quotient by an explicit Serre subcategory of the category of quiver representations (that is, $\mathbb{F}\Gamma$-modules)
via a new colimit model for $M\otimes_{\mathbb{F}\Gamma} L$.
\end{abstract}



\section{Introduction}

The main goal of this article is to investigate the relationship between the module categories of Leavitt path algebras and the module categories of path algebras, and give a criterion for the existence of a nonzero finite dimensional module over Leavitt path algebras.\\


Even though we work with finitely separated digraphs and non-separated row-finite digraphs, to keep the discussion less technical we will restrict 
attention to non-separated finite digraphs in this introduction. Before describing the contents of each section below we provide some context, background and motivation.\\


Leavitt \cite{lea62} defined $L(1,n)$ as the $\mathbb{F}$-algebra generated by $ X_0 , \>  X_1,  \cdots,  \> X_{n-1}, $ $ Y_0, \>   Y_1,  \cdots,  \> Y_{n-1}$  subject to the relations $Y_iX_j=\delta_{ij}$ for $0\leq i ,\> j< n$  and $ X_0Y_0+ X_1 Y_1 +\cdots+ X_{n-1}Y_{n-1} =1$ where $\delta$ is the Kronecker delta. He proved that $L(1,n) $ is a simple algebra and $L(1,n) \cong L(1,n)^n$ but $L(1,n) \ncong L(1,n)^j$ for $j=2, \cdots , n-1$.
 The algebra $L(1,n)$ is the Leavitt path algebra of $R_n$, the rose with $n$ petals:
 $$   
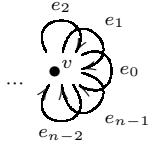
\xymatrix{ & {\bullet^v} \ar@(ur,dr) ^{e_0} \ar@(u,r) ^{e_1}
\ar@(ul,ur) ^{e_2}  \ar@(dr,dl) ^{e_{n-2}} \ar@(r,d) ^{e_{n-1}}
\ar@{}[l]^{\cdots} }$$

The Leavitt path algebra $L(\Gamma)$ of a di(rected )graph $\Gamma$ was defined (many decades after Leavitt's seminal work, via a detour through functional analysis) by Abrams, Aranda Pino \cite{aa05} and by Ara, Moreno, Pardo \cite{amp07} (independently and essentially simultaneously) as an algebraic analog of a graph $C^*$-algebra. It is a universal (Cohn) localization of the path algebra $\mathbb{F} \Gamma$ of the digraph $\Gamma$ \cite[Corollary 4.2]{ab10}. The excellent survey \cite{abr15} is the definitive reference for the history and development of Leavitt path algebras. We will give the precise definition of $L(\Gamma)$ in the next section.   \\

 A major theme in the theory of Leavitt path algebras is to establish a dictionary between the graph theoretic properties of $\Gamma$ and the algebraic structure of $L(\Gamma)$ (see \cite{abr15}, \cite{aas} and their references). In particular, in strictly ascending order of generality, for a finite digraph $\Gamma$ it is known that:
 
  \begin{enumerate}
\item $L(\Gamma)$ has DCC (Descending Chain Condition) on right (or left) ideals \cite[Theorem 2.6]{aaps}  if and only if $\Gamma$ is acyclic (that is, $\Gamma$ has no directed cycles) if and only if $L(\Gamma)$ is von Neuman regular \cite[Theorem 1]{ar10} if and only if $L(\Gamma)$ is finite dimensional if and only if $L(\Gamma)$ is isomorphic to a direct sum of matrix algebras (over the ground field $\mathbb{F}$) \cite[Corollaries 3.6 and 3.7]{aas07}. 
\item $L(\Gamma)$ has ACC (Ascending Chain Condition) on right (or left) ideals \cite[Theorem 3.8]{aaps} if and only if the cycles of $\Gamma$ have no exits if and only if $L(\Gamma)$ is \textit{locally finite dimensional} (i.e., a graded algebra with each homogeneous summand 
being finite dimensional) if and only if $L(\Gamma)$ is a principal ideal ring \cite[Proposition 17]{amp17} in which case $L(\Gamma)$ is isomorphic to a direct sum of matrix algebras over $\mathbb{F} $ and/or matrix algebras over $\mathbb{F}[x,x^{-1}]$ (the Laurent polynomial algebra) \cite[Theorems 3.8 and 3.10]{aas08}.   
\item $L(\Gamma)$ has finite GK (Gelfand-Kirillov) dimension, equivalently $L(\Gamma)$ has polynomial growth if and only if the cycles in $\Gamma$ are mutually disjoint \cite[Theorem 5]{aajz12} if and only if all simple $L(\Gamma)$-modules are finitely presented \cite[Theorem 4.5]{ar14}. \\

In fact (i) and (ii) are special cases of (iii): $\Gamma$ is acyclic if and only if the GK dimension of $L(\Gamma)$ is 0. The digraph $\Gamma$ has a cycle but the cycles of $\Gamma$ have no exits if and only if the GK dimension of $L(\Gamma)$ is 1. The first instance of $L(\Gamma)$ with GK dimension $>1$ is given by the Toeplitz digraph\\\\
$ \Gamma: \qquad  \xymatrix{  \> \> { \bullet}
\ar@(ul,ur) } \xymatrix{ \longrightarrow  { \bullet} } \quad $  (\cite{jac50}, \cite{aajz13}, see Example \ref{Toeplitz} below). \\

We can add the following to this list:
 
 \item $L(\Gamma)$ has a nonzero finite dimensional quotient if and only if  $\Gamma$ has a sink or a cycle such that there is no path from any other cycle to it (Theorem \ref{maximal}) if and only if $L(\Gamma)$ has UGN (Unbounded Generating Number) \cite[Theorem 3.16]{anp} 
 if and only if $L(\Gamma)\oplus L(\Gamma)$ is not a quotient of $L(\Gamma)$ (Corollary \ref{ugn}).\\
 
 Our Corollary \ref{reversible} states: If $L(\Gamma)$ has finite Gelfand-Kirillov dimension then  $L(\Gamma)$ has a nonzero finite dimensional quotient and if $L(\Gamma)$ has a nonzero finite dimensional quotient then $L(\Gamma)$ has IBN. Neither of these implications is reversible.\\
 
 \end{enumerate}

 
 
 Here is a summary of the contents of the rest of this paper: We review the relevant definitions and basic facts in the next section. In section 3, we work in the category  of unital modules over the Leavitt path algebra of a finitely separated digraph $\Gamma$. After observing that this category is equivalent to a subcategory of the category of quiver representations of $\Gamma$ 
  (Proposition \ref{teorem}) we illustrate this point of view with several propositions and examples in sections 3 and 5, providing new proofs of slight extensions of some basic results. (Such as Propositions \ref{loc} and \ref{1-1} below.) Also, in Example \ref{Toeplitz} we give a short proof of the non-splitting theorem in \cite{aajz13}. \\
  
  In section 3 we also give a necessary and sufficient criterion for the existence of a nonzero finite dimensional quotient in terms of dimension functions (Corollary \ref{cor}). While this criterion is still difficult to check in the generality of separated digraphs, in the non-separated case it is equivalent to the existence of a sink or a cycle to which only finitely many vertices can connect but no other cycle. \\

  In section 4 we collect a few definitions and facts needed later. We also reinterpret the criterion for the existence of a nonzero finite dimensional representation in terms of the nonstable $K$-theory of the Leavitt path algebra. \\
  
  From section 5 on we focus on non-separated digraphs. Now finitely separated is the same as row-finite, that is,  no vertex may emit infinitely many arrows. We give an explicit graded projective resolution of $L(\Gamma)$ as an $\mathbb{F}\Gamma$-module in Fact \ref{resolution}.
 The rest of this section is mostly about slight extensions of important examples of  $L(\Gamma)$-modules and their properties considered from our viewpoint.\\
  
  In section 6 we determine all possible finite dimensional quotients of $L(\Gamma)$ for a row-finite digraph $\Gamma$:  Any finite dimensional quotient of $L(\Gamma)$ is isomorphic to $ \oplus M_{n_k} (B_k)$ where the sum is over maximal sinks and maximal cycles with finitely many predecessors in $\Gamma$ and $n_k$ is the number of paths in $\Gamma$ terminating at the relevant sink or at a chosen vertex on the relevant cycle. The cyclic algebra $B_k$ is $\mathbb{F}[x] / (P_k(x))$ with $P_k(0)=1$. If $k$ corresponds to a sink then  $B_k=\mathbb{F}$ if this sink is in the support of $M$, otherwise $B_k=0$ (Theorem \ref{teo1}). Theorem \ref{maximal} states that $L(\Gamma)$ has a nonzero finite 
  dimensional quotient if and only if $\Gamma$ has a maximal sink or a maximal cycle. This criterion (for a finite $\Gamma$) is equivalent to the one
  given in \cite[Theorem 3.16]{anp} for $L(\Gamma)$ to have UGN. In Corollary \ref{ugn} we give a short proof of a generalization: For a Leavitt path algebra not having UGN is equivalent to $L(\Gamma)\oplus L(\Gamma)$ being a quotient of $L(\Gamma)$ (Corollary \ref{ugn}).\\

 Section 7 is about the relationship of the module categories $\mathfrak{M}_{\mathbb{F} \Gamma}$ and $\mathfrak{M}_{L( \Gamma)}$ of the path algebra $\mathbb{F} \Gamma$ and the Leavitt path algebra $L(\Gamma)$. We know from Proposition \ref{teorem} that $\mathfrak{M}_{L( \Gamma)}$ is (essentially) 
 a full subcategory of $\mathfrak{M}_{\mathbb{F} \Gamma}$. In fact $L(\Gamma)$ is a retract (or a summand) of $\mathfrak{M}_{\mathbb{F} \Gamma}$ via the functor $\underline{\>\>\>}\otimes_{\mathbb{F} \Gamma} L(\Gamma)$ by Theorem \ref{forgetful}. We can also realize 
 $\mathfrak{M}_{L( \Gamma)}$ as a localization/quotient of $\mathfrak{M}_{\mathbb{F} \Gamma}$. We identify the relevant Serre subcategory   
$\widehat{\mathfrak{M}}_{\Gamma} $ and show that $\mathfrak{M}_{\mathbb{F}\Gamma}/ \widehat{\mathfrak{M}}_{\Gamma} $ is equivalent to $\mathfrak{M}_{L(\Gamma)}$ (Theorem \ref{Serre}). The proof uses a new realization of $M \otimes L(\Gamma)$ as a direct limit. The fact that $L(\Gamma)$ is a flat 
$\mathbb{F} \Gamma$-module is also relevant. This was proven in \cite{ab10} for a finite $\Gamma$, we give a short proof using the direct limit model for $M \otimes L(\Gamma)$ in Lemma \ref{flat}. 
 \section{Preliminaries}




A {\em di}({\em rected} ){\em graph} $\Gamma$ is a four-tuple $(V,E,s,t)$ where $V$ is the set of vertices, $E$ is the set of arrows, $s$ and $t:E \longrightarrow V$ are the source and  the target functions. The digraph $\Gamma$ is \textit{finite} if $E$ and $V$ are both finite. $\Gamma$ is \textit{row-finite} if $s^{-1}(v)$ is finite for all $v$ in $V$. Given $V' \subseteq V$ the {\em induced subgraph} on $V'$ is $\Gamma':= (V',E',s',t')$ with $E':=  s^{-1}(V') \cap t^{-1}(V')$ ; $\> s':= s\vert_{E'}$ ; $\> t':= t \vert_{E'}$. A subgraph is {\em full} if it is  the induced subgraph on its vertices.\\

A vertex $v$ in $V$ is a \textit{sink} if $s^{-1}(v)= \emptyset$; it is a \textit{source} 
if $t^{-1}(v)= \emptyset$. An isolated vertex is both a source and a sink. 
If $t(e)=s(e)$ then $e$ is  \textit{a loop}. 
A path of length $n>0$ is a sequence $p =e_{1}\ldots e_{n}$ such
that $t(e_{i})=s(e_{i+1})$ for $i=1,\ldots ,n-1.$ The source of $p$ is $s(p
):=s(e_{1})$  and the target of $p$ is $t(p ):=t(e_{n})$. A path $p$ of length 0 consists of a single vertex $v$ where $s(p) :=v $ and $t(p) := v$. We will denote the length of $p$ by $l(p)$.  A path $C=e_1e_2 \cdots e_n$ with $n>0$ is a \textit{cycle} if $s(C )=t(C )$ and $s(e_{i})\neq s(e_{j})$ for $i\neq j$. An arrow $e \in E$ is an \textit{exit} of the cycle $C=e_1e_2 \cdots e_n$ if there is an $i$ such that $ s(e)=s(e_i)$ but $e\neq e_i$. The digraph $\Gamma$ is \textit{acyclic} if it has no cycles. An \textit{infinite path} is an infinite sequence of arrows $e_1e_2e_3 \cdots$ such that  $t(e_k)=s(e_{k+1})$ for $k=1,2,3, \cdots$.

 \begin{remark}
 A digraph is also called an "oriented graph" in graph theory, a "diagram" in topology and category theory,  a "quiver" in representation theory, usually just a "graph" in $C^*$-algebras and Leavitt path algebras. The notation above for a digraph is standard in graph theory. However $Q=(Q_0,Q_1, s,t)$ is more common in quiver representations while $E=(E^0, E^1,s,r)$ is mostly used in graph $C^*$-algebras and in Leavitt path algebras. We prefer the graph theory notation which involves two more letters but no subscripts or superscripts. As in quiver representations we view $\Gamma$ as a small category, so "arrow" is preferable to "edge", similarly for "target" versus "range".
 \end{remark}

There is a preorder defined on the set of sinks and cycles in $\Gamma$: we say that a cycle $C$ \textbf{connects to} a sink $w$  denoted by $C \leadsto w$ if there is a path from $C$ to $w$. Similarly $C\leadsto D$  if there is a path from the cycle C to the cycle D. This is a partial order if and only if the cycles in $\Gamma$ are mutually disjoint. A cycle is \textbf{minimal} with respect to $\leadsto$ if and only if it has no exit (sinks are always minimal). A cycle $C$ is \textbf{maximal} if no other cycle connects to $C$ (in particular, a maximal cycle is disjoint from all other cycles). A sink $w$ is maximal if there is no cycle $C$ which connects to $w$. \\

Given a digraph $\Gamma,$ the \textit{extended digraph} of $\Gamma$ is $\tilde{\Gamma} := (V,E \sqcup E^*, s, t)$ where $E^* :=\{e^*~|~ e\in E \}$ and the functions $s$ and $t$ are  extended as $s(e^{\ast}):=t(e), \> t(e^{\ast }):=s(e) $ for all $e \in E$. Thus the dual arrow $e^*$ has the opposite orientation of $e$. We want to extend $*$ to an operator defined on all paths of $\tilde{\Gamma}$: Let $v^*:=v$ for all $v$ in $V$, $(e^*)^*:=e $ for all $e$ in $E$ and $p^*:= e_n^* \ldots e_1^*$ for a path $p=e_1 \ldots e_n$ with $e_1, \ldots , e_n $ in $E \sqcup E^* $. In particular $*$ is an involution, i.e., $**=id$. \\

A {\em separated digraph} is a pair $(\Gamma, \Pi)$ where $\Gamma =(V,E,s,t)$ is a digraph and $\Pi$  is a partition of $E$ finer than $\left\{ s^{-1}(v): v \in V ~~\text{with} ~s^{-1}(v)\neq \emptyset \right\} $. That is, if $e$ and $f$ are in $X\in \Pi$ then $s(e)=s(f)$. Hence the induced source function $s : \Pi \longrightarrow V$ is well-defined. We will also denote by $X$ the function $E \rightarrow \Pi$ assigning to each arrow $e$ the unique part $X \in \Pi$ containing $e$.\\

If $\Gamma'=(V',E') $ is a subgraph of the separated digraph $\Gamma$ then $\Gamma'$ is also a separated digraph with the separation $\Pi':=\{ X\cap E' \> \vert \> \>  X \in \Pi \>, \> X \cap E'\neq \emptyset \}$. A separated digraph is \textit{finitely separated }if $X$ is finite for all $X$ in $\Pi$. 
Clearly, a subgraph of a finitely separated digraph is also finitely separated. For a non-separated digraph finitely separated is the same as row-finite, that is, $s^{-1}(v)$ is finite for every vertex $v$. 


The {\em Leavitt path algebra of a separated digraph} $(\Gamma, \Pi )$ with coefficients in the field $\mathbb{F}$, as defined in \cite{ag12}, is the $\mathbb{F}$-algebra $L_{\mathbb{F}}(\Gamma , \Pi)$ generated by $V \sqcup E \sqcup E^*$ satisfying:\\
\indent(V)  $\quad \quad \quad ~ vw ~=~ \delta_{v,w}v$ for all $v, w \in V ,$ \\
\indent($E$)  $\quad \quad \quad s ( e ) e  = e=e t( e)  $ for all $e \in E\sqcup E^*$, \\
\indent(SCK1) $\quad ~e^*f ~ = ~ \delta_{e,f} ~t(e)$ for all $e,f\in X$ and all $X\in \Pi $, \\
\indent(SCK2) $\quad ~s (X)~ = ~ \sum_{e\in X} ee^*$  for every finite $X \in \Pi$\\
where $\delta$ is the Kronecker delta. \\


We will usually suppress the subscript $\mathbb{F}$ when we denote our algebras. When $\Gamma$ or $\Pi$ are clear from the context we may also omit these from our notation. We will also abbreviate $s(e)$, $t(e)$ and $s(X)$ etc., as $se$, $te$ and $sX$ to reduce notational clutter.\\

The relations (V) simply state that the vertices are mutually orthogonal idempotents. If we only impose the relations (V) and ($E$) then we obtain $\mathbb{F}\tilde{\Gamma}$, the \textit{path }(or \textit{quiver}) \textit{algebra} of the extended digraph $\tilde{\Gamma}$ : The paths in $\tilde{\Gamma}$ form a vector space basis of $\mathbb{F} \tilde{\Gamma}$, the product $pq$ of two paths $p$ and $q$ is their concatenation if $t p=s q$ and 0 otherwise.  We get the\textit{ Cohn path algebra} $C(\Gamma,\Pi)$ \textit{of the separated digraph} $(\Gamma,\Pi)$ when we impose the relations (SCK1) in addition to (V) and ($E$). Hence $L(\Gamma,\Pi)$ is a quotient of $C(\Gamma,\Pi)$, which is a quotient of $ \mathbb{F}  \tilde{\Gamma}$. The abbreviation SCK stands for Separated Cuntz-Krieger. \\


Note that $L_{\mathbb{F}}(\Gamma)$ is not a quotient of the polynomial algebra in the noncommuting variables $V \sqcup E \sqcup E^*$ because we need to consider the algebra of polynomials without a constant term. In particular, when $\Gamma$ is a single vertex $v$ with no arrows then $L_{\mathbb{F}}(\Gamma) = \mathbb{F}v \cong \mathbb{F}$ not $\mathbb{F} \oplus \mathbb{F}v$ (Similarly for the path algebra  and also for the Cohn path algebra).   \\

The algebras $\mathbb{F} \Gamma$, $ \mathbb{F}\tilde{\Gamma}$, $C(\Gamma, \Pi)$ and $L(\Gamma, \Pi)$ have 1 if and only if $V$ is finite, in which case the sum of all the vertices is the unit: It is clear that $\sum_{v \in V} v=1$ when $V$ is finite. For the converse, a given element in any these algebras is a finite linear combination of paths in $\tilde{\Gamma}$ and we can pick $v \in V$ which is not the source of any of these paths if $V$ is infinite. Now left multiplication by $v$ gives zero, so there is no unit element in any of these algebras, since Proposition \ref{1-1} below shows that $v\neq 0$ in $L(\Gamma, \Pi)$ for every $v \in V$, hence also in $C(\Gamma, \Pi)$.\\ 

When the vertex set $V$ is infinite, $L$ does not have 1 but $L$ has \textit{local units}: a commuting set of idempotents $\Omega$ such that for all $a, b$ in the ring there is a $\omega$ in $\Omega$ with $\omega a=a=a \omega $  and $\omega b=b=b\omega$. In $L$, as well as in $\mathbb{F}\Gamma$ and $C(\Gamma, \Pi)$, we may take $\Omega$ to be all finite sums of vertices. We call a ring homomorphism $\varphi: R\longrightarrow S$ \textit{unital} if $R$ has a set of local units $\Omega$ with $\varphi (\Omega)$ being a set of local units in $S$. With this definition the natural homomorphisms between $\mathbb{F}\Gamma$, $ C(\Gamma, \Pi)$ and $L$ are all unital.\\

When the ring has a 1 this definition of a unital ring homomorphism is equivalent to $\varphi(1)=1$ in $S$: If $R$ has 1 then there is a $\omega \in \Omega$ with $1=\omega 1=\omega \in \Omega$, hence $1\in \Omega$. If $\varphi$ is unital then there is $\omega \in \Omega$ with $\varphi (\omega)=1$ since $1 \in \varphi(\omega)$. Now $\varphi(1) =\varphi(1)1=\varphi (1)\omega =\varphi(\omega) =1$. Hence our definition of a unital ring homomorphism implies the standard definition when $R$ and $S$ have 1. For the converse we may take $\Omega =\{ 1 \}$. \\

If an algebra $A$ does not have 1 then the ideal $I$ generated by $U \subseteq A$  is  $\mathbb{F} U+ AU + UA+ AUA $, in general.
If $A$ has local units then this simplifies to $I=AUA $.\\

There is a $\mathbb{Z}$-grading on $ \mathbb{F}  \tilde{\Gamma}$ and all the other algebras above given by $\lvert v \vert =0$ for $v$ in $V$, $\lvert e\rvert =1$ and $\lvert e^* \rvert =-1$ for $e$ in $E$. This defines a grading on all our algebras since all the relations are homogeneous. The linear extension of $*$ on paths induces a grade-reversing involutive anti-automorphism (i.e., $\lvert \alpha^*\rvert =- \lvert \alpha \rvert $ and $(\alpha \beta)^*=\beta^* \alpha^*$). Hence these algebras are $\mathbb{Z}$-graded $*$-algebras and the (graded) categories of left modules and right modules for any of these algebras are equivalent.\\

More generally, we may consider $G$-gradings on $L(\Gamma, \Pi)$ for any group $G$, with $V \sqcup E$ being homogeneous. 

\begin{lemma} \label{hom}
If $G$ is a group and $\lvert \> \rvert$ is a $G$-grading on $L(\Gamma, \Pi)$ such that all vertices $v$ in $V$ and all arrows $e$ in $E$ are homogeneous then all $e^*$ in $E^*$ are also homogeneous with $\lvert e^* \rvert =\lvert e \rvert ^{-1}$.
\end{lemma}
\begin{proof}
We see that $\lvert v \lvert =1$ for all $v$ in $V$ since $v^2=v$. Let $e^*=\sum e^*_g$ where $e^*_g$ is the homogeneous component of $e^*$ of grade $g$. Hence $e^*_g e=0$ if $g\neq \lvert e \rvert^{-1}$ since $e^*e=te$ and $\vert te \vert =1 $. 
Similarly, if $e \in X$ then $e^*_gf=0$ for $e \neq f \in X$. 
 Thus $e_g^*sX=e_g^* \sum_{f \in X} ff^*=0$ if $g \neq  \lvert e \rvert ^{-1}$. Also $e_g^*v=0$ if $v\neq sX=se$ because $e^*v=0$. Since $e_g^* v=0$ for all $v \in V$ when $g\neq  \lvert e \rvert ^{-1}$ we see that $e^*$ is homogeneous and $ \lvert e ^*\rvert =\lvert  e \rvert ^{-1}$.
\end{proof} \\

Consequently, any function from $ E $ to $G$ defines a unique $G$-grading on $L(\Gamma, \Pi)$ with  $\lvert v \rvert =1$ for all $v \in V$ and  $\lvert e^* \rvert = \lvert e \rvert ^{-1}$ since the relations (CK1) and (CK2) are homogeneous. A morphism (or a refinement) from a $G$-grading to an $H$-grading on the algebra $A$ is given by a group homomorphism $\phi:G \longrightarrow H$ such that for all $h \in H$, $A_h= \oplus_{\phi (g)=h} A_g$ where $A_g:= \{a \in A  :\lvert a\rvert_{_{G}}=g \} \cup \{0\} $.
 There is a universal $G$-grading on $L(\Gamma, \Pi)$ and $C(\Gamma, \Pi)$ which is a refinement of all others:

\begin{proposition}
Let $G:= F_{E}$ be the free group on the set of arrows. The $G$-grading defined by $\vert v\vert_{_G}=1$ and $\vert e\vert_{_G}= e$  is an initial (universal) object in the category of $G$-gradings of 
$\mathbb{F}\Gamma$ or $L(\Gamma, \Pi)$ or $C(\Gamma, \Pi)$ with $V \sqcup E$ being homogeneous.
\end{proposition}
\begin{proof}
For any $H$-grading let $\phi: G \rightarrow H$ be the homomorphism given by $\phi(e)= \vert e\vert_{_H}$.
\end{proof}\\

Combined with the existence of certain representations of $L(\Gamma, \Pi)$ defined in the next section this universal grading is useful in showing that some elements of $L(\Gamma, \Pi)$ are nonzero 
(or linearly independent) as in Proposition \ref{1-1} below.\\

When $\Pi=\left\{ s^{-1}(v) \mid v \in V,    s^{-1}(v)\neq \emptyset \right\}$, we say that $\Gamma$ is not separated, $C(\Gamma, \Pi)$ is denoted by $C(\Gamma)$ and called the \textit{Cohn path algebra of} $\Gamma $. Similarly $L(\Gamma,\Pi)$ is denoted by $L(\Gamma)$ and called the \textit{Leavitt path algebra of} $\Gamma$.  Also the conditions (SCK1) and (SCK2) are denoted by (CK1) and (CK2) respectively \cite{aa05}, \cite{amp07}. Since $e^*f=e^*(se)(sf)f$ by ($E$), using ($V$) the relation (SCK1) is shortened to: (CK1) $e^*f =\delta_{e,f} te$ for all $e,f \in E$. \\

A subset $H$ of $V$ is {\em hereditary} if for any path $p$, $sp \in H$ implies that $tp \in H$ \cite{aa05}; $H$ is {\em $\Pi$-saturated} if $\{ te : e \in X \} \subseteq H$ for some (finite) $X\in \Pi$ implies that $sX \in H$ \cite{ag12}. If $I$ is an ideal of $L(\Gamma,\Pi)$ and $p$ is a path in $\Gamma$ with $sp \in I$ then $tp=p^*p=p^*(sp)p \in I$, also if $\{ te : e \in X \} \subseteq I$ then $sX=\sum_{e \in X}  ee^*=\sum_{e\in X}  e(te)e^* \in I$, so $I\cap V$ is hereditary and $\Pi$-saturated. We have a Galois connection between the subsets of $V$ and the ideals of $L(\Gamma, \Pi)$ given by $S \mapsto (S)$ and $I \mapsto I\cap V$ which gives a bijection between hereditary saturated subsets of $V$ and graded ideals of $L(\Gamma)$ when $\Gamma$ is a (non-separated) row-finite digraph \cite[Theorem 5.3]{amp07}. 

 \section{Quiver representations and $L(\Gamma, \Pi)$-modules}
We will work in the category $\mathfrak{M}_{L }$ of unital (right) modules over $L:=L_{\mathbb{F}} (\Gamma, \Pi)$. However $L$ has a 1 if and only if the vertex set $V$ is finite. Even if $V$ is infinite, we define a unital $L$-module as a module $M$ with the property that $ML=M$, i.e., for any $m$ in $M$ we can find $\lambda_1, \lambda_2, \ldots , \lambda_n$ in $L$ and $m_1, m_2, \ldots, m_n$ in $M$ so that $m=m_1\lambda_1+m_2 \lambda_2 +\cdots + m_n\lambda_n$. This condition is equivalent to the standard definition of unital (when $L$ has a 1) since $m1=(m_1\lambda_1+m_2 \lambda_2 +\cdots + m_n\lambda_n)1=m_1\lambda_1 1 +m_2 \lambda_2 1+\cdots + m_n\lambda_n1=m_1\lambda_1+m_2 \lambda_2 +\cdots + m_n\lambda_n=m$. The category of unital modules is an abelian category with sums since it is closed under taking quotients, submodules, extensions, (arbitrary) sums (but not infinite products: if $V$ is infinite then the $L$-module $L^V$ is not unital).\\

We will need the following consequence of (SCK1) and (SCK2):

\begin{lemma} \label{lemma} Let $\Gamma$ be a finitely separated digraph and $M$ be a right $L$-module. Then $(\mu_e)_{e \in X}: MsX \longrightarrow  \oplus_{e \in X} Mte$ for all $X \in \Pi$ is an isomorphism (of vector spaces), where $\mu_e$ is right multiplication by $e$. Moreover, if $M$ is unital then $M = \oplus_{v \in V} Mv$ (as a vector space) and $MsX=\oplus_{e \in X}Mee^*$.
\end{lemma}
\begin{proof}
To see that the inverse of $(\mu_e)_{e \in X}$ is $[\mu_{e^*}]_{e \in X}: \oplus_{e \in X} Mte \longrightarrow MsX$ given by $(m_e)_{_{e \in X}} [\mu_{e^*}]_{_{e \in X}}=  \sum_{e \in X} m_e e^*$ we check their compositions:\\
$$(m_e)_{_{e \in X}} [\mu_{e^*}]_{_{e \in X}}(\mu_f)_{_{f \in X}}=( \sum_{e \in X} m_e e^*)(\mu_f)_{_{f \in X}}=( \sum_{e \in X} m_e e^*f )_{_{f \in X}}=\big(  m_f \big)_{_{f \in X}}$$
where the last equality uses the relation $e^*f=\delta_{_{e,f}}tf $ and $m_f ( tf) =m_f$ since $m_f \in Mtf$. Also for $m \in MsX$ we get 

$$m(\mu_e)_{_{e \in X}}[\mu_{e^*}]_{_{e \in X}}=\big( me\big)_{_{e \in X}}[\mu_{e^*}]_{_{e \in X}}=\sum_{e \in X} mee^*=msX=m .$$

When $M$ is unital for any $m$ in $M$ we have $m=\sum_{k=1}^n m_k  \alpha_k = \sum_{j=1}^l m_j' v_j $ for some vertices $v_1, \cdots, v_l \in V$. Hence $M=\sum_{v \in V} Mv$. This sum is direct: For any finite set $\Lambda$ of vertices with $v \notin \Lambda$ if $m \in Mv \cap \sum_{w \in \Lambda} Mw$ then $m=mv$  (since $m \in Mv$) but $mv=0$ (since $wv=0$ for each $w \in \Lambda$). Thus $Mv \cap \sum_{w \in \Lambda } Mw=0$.\\

 $MsX=\sum_{e \in X} Mee^*$ since $sX=\sum_{e \in X}ee^*$. If $ e \neq f$ in $X$ then $ee^*ff^*=e(e^*f)f^* = 0$. Also $ee^*ee^*=ee^*$. That is, $MsX=\oplus_{e \in X}Mee^*$. 
 \end{proof}\\
 
Consequently, the linear transformation defined by right multiplication with any $e$ in $E$ from $Mse$ to $Mte$  is onto. Hence right multiplication with any path $p$ from $Msp$ to $Mtp$ is also onto.  
Similarly, right multiplication with $p^*$ is injective.\\


We want to view the category $\mathfrak{M}_{L}$ as a subcategory of $\mathfrak{M}_{\mathbb{F}\Gamma}$, the category of unital modules over the path algebra $\mathbb{F}\Gamma$, or equivalently, the category of quiver representations of $\Gamma$. The category of quiver representations of $\Gamma$ is the category of functors from the path category of the digraph $\Gamma$ (whose objects are the vertices $V$ and the morphisms are the paths in $\Gamma$) to the category of $\mathbb{F}$-vector spaces. A morphism of quiver representations is a natural transformation between two such functors. That is, a quiver representation $\rho$ assigns a (possibly infinite dimensional) vector space $\rho(v)$ to each vertex $v$ and a linear transformation $\rho(e): \rho(se) \longrightarrow \rho(te)$ to each arrow $e$. A morphism of quiver representations $\varphi: \rho \longrightarrow \sigma$ is a family of linear transformations $\{ \varphi_v : \rho(v) \longrightarrow \sigma(v) \}_{ v \in V}$ such that $\forall e \in E $ the diagram 
$
\begin{array}{cc}
  \rho(se)&  \stackrel{\rho(e)}{\longrightarrow} \rho(te)   \\
  \varphi_{se}\downarrow &   \qquad \quad \downarrow \varphi_{te}   \\
  \sigma(se) &\stackrel{\sigma(e)}{\longrightarrow} \sigma(te)   
\end{array}$ commutes \cite{dw05}.\\\\

In Proposition \ref{teorem} below the hypothesis on $\Gamma$ of being finitely separated may be removed (even in the generality of Cohn-Leavitt path algebras of separated digraphs as defined in \cite{ag12}) at the cost of complicating condition (I). 
We will not pursue this generality here.

	\begin{proposition} \label{teorem}
If $\Gamma= (V,E,s,t,\Pi)$ is a finitely separated digraph then the category $\mathfrak{M}_{L}$
 is equivalent to the full subcategory of quiver representations $\rho$ of $\Gamma$ satisfying: \\

For all $X\in \Pi \> , \quad (\rho(e))_{e \in X}: \rho ( sX) \longrightarrow \bigoplus\limits_{e \in X} \rho(te) \textit{ is an isomorphism.} \quad (I)$
Similarly, the full subcategory of graded quiver representations (with respect to the standard $\mathbb{Z}$-grading where the grade of a  path is its length) satisfying (I) is equivalent to the category of graded unital $L$-modules (with respect to the standard $\mathbb{Z}$-grading). 

\end{proposition}
 \begin{proof}
 Given  a right $L$-module $M$ we define a quiver representation $\rho_{_M}$ as follows: $\rho_{_M} (v) = Mv$ and $\rho_{_M} (e): Mse \longrightarrow Mte$ is defined by $m(se)\rho_{_M} (e):= m(se)e=me(te)$. 
  By the first part of Lemma \ref{lemma} (I) is satisfied. If $\varphi :M \longrightarrow N$ is an $L$-module homomorphism then $\varphi_v$ is the linear transformation making the diagram $
\begin{array}{c}
  \rho_{_M}(v)=Mv  \hookrightarrow M   \\
  \qquad \qquad  \varphi_v\downarrow    \quad \quad \downarrow \varphi  \\
\rho_{_N}(v)=Nv \hookrightarrow N  
\end{array}$ commutative. This defines a homomorphism of quiver representations, i.e., $\rho_{_M}(e)\varphi_{te}= \varphi_{se} \rho_{_N}(e)$  because right multiplication by $e$ commutes with $\varphi$.\\

Given a quiver representation $\rho$ we define (underlying vector space of) the corresponding module 
$M_\rho= \oplus_{_{v \in V}} \rho(v)$. To define $L$-module structure we will use the projections $p_v: \oplus_{_{w \in V}} \rho(w) \longrightarrow \rho(v)$, the inclusions $\iota_v : \rho(v) \longrightarrow \oplus_{_{w \in V}} \rho(w)$ for $v\in V$ and the projections $p_e:\oplus_{_{f \in X_e}} \rho(tf) \longrightarrow \rho(te)$, the inclusions $\iota_e: \rho(te) \hookrightarrow \oplus_{_{f \in X_e}} \rho (tf)$.  Now, let $mv:=mp_v\iota_v \> , \> \> me:= mp_{se} \rho(e) \iota_{te} $, $  me^*:= mp_{te}\iota_{e} (\rho(f))_{_{f \in X_e}}^{-1} \iota_{sX}$. It is routine (albeit tedious) to check that the defining relations of $L$ are satisfied.\\

Now let's check that the constructions above yield equivalences of categories. $M_{\rho_{_M}}:= \oplus_{_{v \in V}} Mv=M $ by Lemma \ref{lemma}, as vector spaces. It is easy to check that the $L$-module structures also match. Also, given a module homomorphism $\varphi : M \longrightarrow N$, we have $\varphi= \oplus_{_{v \in V}} \varphi_v: \oplus_{_{v \in V}} Mv \longrightarrow \oplus_{_{v \in V}} Nv$. For the composition in the other order  $\rho_{_{M_{\rho}}} (v):= M_{_{\rho}}v :=\big(\oplus_{_{w \in V}} \rho(w) \big)v=\rho(v)$ and $\rho (e)= \rho_{_{M_\rho}} (e): M_\rho se \longrightarrow M_\rho te$ because the diagram 
$$\begin{array}{c}
  M_\rho se = \rho(se)  \stackrel{\iota_{se}}{\hookrightarrow} M_\rho := \oplus_{_{w \in V}} \rho (w)    \\
  \qquad \qquad \quad \quad \rho(e) \downarrow     \qquad \qquad \qquad  \qquad   \downarrow p_{se} \rho (e) \iota_{te}   \\
  M_\rho te = \rho(te)  \stackrel{\iota_{te}}{\hookrightarrow} M_\rho := \oplus_{_{w \in V}} \rho (w)   
\end{array} $$ commutes. Finally, for any homomorphism $\big\{ \varphi_v : \rho (v) \longrightarrow \sigma (v) \big\}_{_{v \in V}} $ from $\rho$ to $\sigma$, the $v$-component of $ \oplus_{_{w \in V}} \varphi_w $ is  $ \varphi_v : \rho_{_{M_\rho}} (v)= \rho (v) \longrightarrow \sigma_{_{M_\sigma}} (v) =\sigma (v)$.\\

For the graded version the proof is essentially the same as above. Note that the morphisms $(\rho(e))_{e \in X}: \rho ( sX) \longrightarrow \bigoplus\limits_{e \in X} \rho(te)$ for all $X \in \Pi$ of condition (I) necessarily have grade $+1$, i.e., a homogeneous element of grade $n$ is send to a homogeneous element of grade $n+1$.
 \end{proof}\\

  The argument above almost proves the stronger statement that $\mathfrak{M}_{L}$ is \textit{isomorphic} to the subcategory of quiver representations of $\Gamma$ satisfying the condition (I). The only issue is the difference between internal and external direct sums. In fact, we can obtain an isomorphism of categories if we work in a graded category where each subspace $Mv$ of the $L$-module $M$, $v \in V$, is a homogeneous summand. There is no need for such an artifice since equivalence of categories is sufficient for our purposes.\\
  
  $L=\oplus \> vL$ where the sum is over $v \in V$ and each $vL$ is a cyclic projective $L$-module. The vector space $Mv \cong Hom^L(vL, M)$ is actually a module over the corner algebra $vLv \cong End^L(vL)$. \\
  
$\mathfrak{M}_L$ is not only a full subcategory of $\mathfrak{M}_{\mathbb{F} \Gamma}$ but also a quotient, in fact, a retract of $\mathfrak{M}_{\mathbb{F} \Gamma}$.
  
\begin{theorem} \label{forgetful} The composition of the forgetful functor from $\mathfrak{M}_{L(\Gamma, \Pi)}$ to $\mathfrak{M}_{\mathbb{F} \Gamma}$ with  $\underline{\>\>\> } \otimes_{\mathbb{F}\Gamma} L(\Gamma, \Pi)$ from $\mathfrak{M}_{\mathbb{F} \Gamma}$ to $\mathfrak{M}_{L(\Gamma, \Pi)}$ is naturally equivalent to the identity functor on $\mathfrak{M}_{L(\Gamma, \Pi)}$.
\end{theorem} 
\begin{proof}
We first check that both the forgetful functor and $\underline{\>\>\> } \otimes_{\mathbb{F}\Gamma} L(\Gamma, \Pi)$ send unital modules to unital modules. If $M$ is a unital $L(\Gamma, \Pi)$-module then $M=\oplus_{v\in V} Mv$ by Lemma \ref{lemma}. Since $V \subset \mathbb{F}\Gamma$, $M$ is also unital as an $\mathbb{F}\Gamma$-module. If $M$ is a unital $L(\Gamma, \Pi)$-module then $(M \otimes L(\Gamma, \Pi))L(\Gamma, \Pi)=M \otimes L(\Gamma, \Pi)$ because $L(\Gamma, \Pi)$ is unital as an $L(\Gamma, \Pi)$-module.\\

The $L(\Gamma,\Pi)$-module homomorphism $M\otimes_{\mathbb{F}\Gamma}  L(\Gamma,\Pi) \longrightarrow M$
where $m\otimes \alpha \mapsto m\alpha$ defines a natural transformation from the composition of the forgetful functor with $\underline{\>\>\> } \otimes_{\mathbb{F}\Gamma} L(\Gamma, \Pi)$ to the identity functor when $M$ is an $L(\Gamma, \Pi)$-module. To see that this is an isomorphism we define its inverse $M \longrightarrow 
M\otimes_{\mathbb{F}\Gamma}  L(\Gamma, \Pi)$ as $m \mapsto \sum m \otimes v$ where the sum is over $V_m:=\{v \in V \vert mv \neq 0\}$. This sum is finite since $M$ is unital.
If $u \notin V_m$ then $m \otimes u=0$, hence the sum $\sum m\otimes v$ may be taken over any finite subset of $V$ containing $V_m$.\\

To see that this is $L(\Gamma, \Pi)$-linear it is enough to check the action of the generators $V \sqcup E \sqcup E^*$. For all 
$u \in V$ and all $m \in M$ we have $\sum mu \otimes v=m \otimes u=(\sum m\otimes v)u$
because $V$ is a set of orthogonal idempotents. For all $e \in E$ and all $m \in M$ we have $\sum me \otimes v=me \otimes te=m \otimes e=(\sum m \otimes v)e$ because $ve\neq 0$ iff $v=se$. For all $e^* \in E^*$ and all $m \in M$ we have $\sum me^*\otimes v= me^*\otimes se= m e^*\otimes \sum_{f \in X} ff^* =\sum me^*f \otimes f^*=m (te) \otimes e^*=m \otimes e^*=(\sum m\otimes v)e^*$ using (SCK2) and (SCK1) where $e \in X \in Pi$.
\\

The composition $m \mapsto \sum m\otimes v \mapsto \sum mv=m $ by Lemma \ref{lemma}. Since $M=\oplus Mv$ and $Path(\tilde{\Gamma})$ spans $L(\Gamma,\Pi)$, elements of the form $m \otimes v$ with $m \in Mv$ generate $M\otimes L(\Gamma,\Pi)$ as an $L(\Gamma,\Pi)$-module. For such elements $m \otimes v \mapsto mv \mapsto mv \otimes v=m\otimes v$. That is, the other composition is also identity.
\end{proof}\\

  $\mathfrak{M}_{L}$ is a categorical localization of the quiver representations
 of $\Gamma$. This is related to the fact that the algebra $L$ is a universal localization of the path algebra $\mathbb{F} \Gamma$, when $\Gamma$ is non-separated this is Corollary 4.2 of \cite{ab10}. Recall that the universal localization $\Sigma^{-1}A$ of an algebra $A$ with respect to a set $\Sigma=\{ \sigma : P_{\sigma} \longrightarrow Q_{\sigma} \}$ of homomorphisms between finitely generated projective $A$-modules, is an initial object among algebra homomorphisms $f: A \longrightarrow B$ such that $\sigma \otimes id : P_{\sigma} \otimes_A B \longrightarrow  Q_{\sigma} \otimes_A B$ is an isomorphism for every $\sigma$ in $\Sigma$.
 
 \begin{proposition} \label{loc}
 If $\Gamma$ is a finitely separated digraph then $L(\Gamma, \Pi)$ is the universal localization of $\mathbb{F}\Gamma$ with respect to $\{ \sigma_X :  \oplus_{e \in X} (te)\mathbb{F}\Gamma \longrightarrow  (sX)\mathbb{F}\Gamma \> \> \vert \>  \> X \in \Pi \}$
 where $\sigma_X \big( (a_e)_{e \in X} \big) : = \sum_{e \in X}  ea_e$.
 \end{proposition}
 \begin{proof}
 For any $v \in V$ the cyclic (right) module $v\mathbb{F}\Gamma$ is projective since $v$ is an idempotent. For all $X$ in $\Pi$, $\> \sigma_X \otimes id_{L(\Gamma, \Pi)} $ is an isomorphism with inverse $(e^*\cdot)_{e \in X}$ where $e^*\cdot $ denotes left multiplication by $e^*$. When $f: \mathbb{F}\Gamma \longrightarrow B$ is an algebra homomorphism  $f(v)^2=f(v)$ and $v\mathbb{F}\Gamma \otimes_{\mathbb{F}\Gamma} B \cong f(v)B$ via $a\otimes b \mapsto f(a)b$ and $b \mapsto v\otimes b$ (note that $\mathbb{F}\Gamma$ and $B$ need not be unital).  If $f: \mathbb{F}\Gamma \longrightarrow B $ is an algebra homomorphism such that $\sigma_X \otimes id_B$ is an isomorphism for all $X$ in $\Pi$ then the composition $f(sX)B \cong (sX)\mathbb{F}\Gamma \otimes_{\mathbb{F}\Gamma} B \stackrel{\sigma_X^{-1}}{\longrightarrow} \Big(\oplus (te') \mathbb{F}\Gamma \Big) \oplus_{\mathbb{F}\Gamma} B \cong \oplus f(te')B \stackrel{pr_e}{\longrightarrow} f(te)B$ is uniquely and completely determined by the image of $f(sX)$, which we call $f(e^*)$. Now $\tilde{f}(v):=f(v)$ for all $v$ in $V$, $\tilde{f}(e):= f(e)$ for all $e$ in $E\sqcup E^*$ defines the unique homomorphism $\tilde{f}: L(\Gamma, \Pi) \longrightarrow B$ factoring $f$ through $\mathbb{F}\Gamma \longrightarrow L(\Gamma, \Pi)$.
  \end{proof}\\

 With the quiver representation viewpoint there is no need to  mention the generators $\{e^* : e \in E \}$ explicitly, they are implicit in the condition (I). Theorem \ref{teorem} also enables us to construct concrete models for $L$-modules and homomorphisms between them as illustrated in the following applications.
  
  \begin{proposition} \label{prop}
  Let $\Gamma$ be a finitely separated digraph. If $d: V \longrightarrow \mathbb{N} \cup \{ \infty \}$ satisfies $d(sX)=\sum_{_{e \in X} } d(te)$ for all $X$ in $\Pi$ then there is an $L$-module $M$ with $dim^{\mathbb{F}}(Mv)=d(v)$.  
 \end{proposition}
     
       \begin{proof}
  Let the quiver representation $\rho$ be given by $\rho(v) :=  \mathbb{F}^{d(v)}$ if $ d(v) < \infty$   and $\rho(v):= \mathbb{F}^{(\mathbb{N})} $ otherwise. We can find isomorphisms $\theta_{_X}: \rho(sX) \longrightarrow \oplus_{_{e \in X}} \rho(te)$ for all $X$ in $\Pi$ by the hypothesis on $d$. Let $\rho(e): =\theta_{_X} pr_e$ for all $e$ in $E$. Condition (I) is satisfied by construction and the corresponding $L$-module $M$ of Theorem \ref{teorem} has $dim^{\mathbb{F}}(Mv)=dim^{\mathbb{F}} \rho(v)=d(v)$.
      \end{proof} 
 \begin{corollary} \label{cor1}     
 There is an $L$-module $M$ with $Mv \cong \mathbb{F}^{(\mathbb{N} )}$ for all $v$ in $V$. Hence $p$ and $p^*$ are nonzero in $L$ for every path $p$ of $\Gamma$. 
 \end{corollary}
 \begin{proof}
 The existence of an $L$-module $M$ with $Mv \cong \mathbb{F}^{(\mathbb{N} )}$ for all $v$ in $V$ is given by Proposition \ref{prop}. Right multiplication by $p$ is onto from $Msp$ to $Mtp \cong \mathbb{F}^{(\mathbb{N} )}$ by Lemma \ref{lemma}. Hence every path $p$ in $L$ is nonzero. Since $*$ is an involution $p^*$ also is nonzero.
 \end{proof}\\

The  next proposition shows that the natural algebra homomorphism from $\mathbb{F}\Gamma$ to $L(\Gamma, \Pi)$ is injective. (For a non-separated digraph this is Lemma 1.6 of \cite{goo09}.) Thus we may regard the path algebra $\mathbb{F}\Gamma$ as a subalgebra of the Leavitt path algebra $L$.  This is still true when $\Gamma $ is not finitely separated and it follows from Theorem 2.7 in \cite{ag12}, a more general result giving an $\mathbb{F}$-basis for a Cohn-Leavitt path algebra. Below we provide a different (short) proof.  

\begin{proposition}\label{1-1}
If $\Gamma$ is a finitely separated digraph then the homomorphism from the path algebra $\mathbb{F}\Gamma$ to the Leavitt path algebra $L$ is injective.
\end{proposition}
\begin{proof}
Let $F_E$ be the free group on $E$ (the arrow set of $\Gamma$). We can define $F_E$-gradings 
on  $\mathbb{F}\Gamma$ and $L$ by $\vert v \vert =1$ for all $v \in V$, $\vert e \vert =e$ and $ \vert e^* \vert =e^{-1}$ for all $e \in E$ (since all the relations are homogeneous this grading is well-defined). The homomorphism $\mathbb{F}\Gamma \longrightarrow L$ is graded so its kernel is a graded ideal. Any homogeneous element of  $\mathbb{F}\Gamma$ is either a scalar multiple of some path $p$ of positive length or a linear combination of vertices. By Proposition \ref{prop} there is an $L$-module $M$ with $dim(Mv)=\infty$ for every $v$ in $V$. Vertices of $\Gamma$ are orthogonal idempotents of $L$ defining projections on $M$ with infinite dimensional images, so a linear combination of vertices will be zero in $L$ if and only if it is trivial. Also, Lemma \ref{lemma} implies that the linear transformation given by right multiplication with $p$ of positive length from $Msp$ to $Mtp$ is onto, hence $p \neq 0$ in $L$. 
Therefore the kernel of $\mathbb{F}\Gamma \longrightarrow L$ is trivial.  
\end{proof}\\

Lemma 1.6 of \cite{goo09}  actually states that the set of all paths and all dual paths $\{p \} \cup \{ p^* \}$ is linearly independent in $L$ for a non-separated digraph $\Gamma$. The proof above yields this stronger statement for a finitely separated $\Gamma$ since all elements of $\{p \} \cup \{ p^* \}$ are nonzero (by Corollary \ref{cor1}) homogeneous and they have different grades. Theorem 2.7 in \cite{ag12} provides an even stronger statement.

  \begin{definition}

A \textbf{dimension function} of a finitely separated digraph $\Gamma$ is a function $d: V  \longrightarrow \mathbb{N}$  satisfying $d(sX) =\sum_{e\in X} d(te) \quad \textit{ for all } X \textit{ in } \Pi.$

\end{definition}
 
If the $L$-module $M$ is \textit{finitary}, that is,  $dim(Mv)<\infty$ for all $v$ in $V$ then Lemma \ref{lemma} shows that $d(v):=dim(Mv)$ is a dimension function. The converse also holds, that is, every dimension function is realizable:

\begin{corollary} \label{cor}
If $d$ is a dimension function for $\Gamma$ then there exists an $L$-module $M$ with $dim(Mv)=d(v)$. Hence $L$ has a nonzero finite dimensional module if and only if $\Gamma$ has a nonzero dimension function of finite support.
\end{corollary}

\begin{proof} 
This is the special case of Proposition \ref{prop} with $d: V \longrightarrow \mathbb{N}$. By Lemma \ref{lemma}, $dim(M)=\sum_{v \in V} dim(Mv)$, hence $d(v)=dim (Mv)$ has finite support if $M$ is finite dimensional. 
\end{proof}\\

If $M$ is a nonzero finite dimensional $L$-module then the image of $L$ in 
$End^\mathbb{F} M$ is (isomorphic to) a nonzero finite dimensional quotient of $L$.
Conversely a nonzero finite dimensional quotient is also a (nonzero finite dimensional) $L$-module. 
Thus the corollary above gives a necessary and sufficient condition for the existence of nonzero finite dimensional quotient. When $\Gamma$ is a (non-separated) row-finite digraph we determine all possible finite dimensional quotients of $L(\Gamma)$ in Theorem \ref{teo1} of Section 6.



\section{Support subgraphs and the monoid of a finitely separated digraph}

In this section we reinterpret the criterion for the existence of a nonzero finite dimensional representation in terms of the nonstable $K$-theory of $L(\Gamma): =L_{\mathbb{F}}(\Gamma, \Pi)$. First we will need a few definitions. \\
 
A subgraph $\Gamma'=( V',E')$ of $\Gamma=( V, E)$ is called \textbf{cohereditary} when for all $e$ in $E$ if $te \in V'$ then $se$ also is in $V'$. 
A subgraph $\Gamma'$ of the separated digraph $\Gamma$ is \textbf{colorful} for any $X$ in $\Pi$ if $sX \in V'$ then $ X \cap E' \neq \emptyset $. If $M$ is a right $L(\Gamma)$-module  then the \textbf{support subgraph} of $M$, denoted by $\Gamma_{_M}$, is the induced subgraph of $\Gamma$ on $V_{_M}:=\{v \in V \mid Mv \neq 0 \}$.  \\
 

The subgraph $\Gamma'=(V',E')$ is cohereditary  if and only if $V\setminus V'$ is a hereditary subset of $V$. When $\Gamma' $ is full then $\Gamma'$ is colorful if and only if $V \setminus V'$ is $\Pi$-saturated as defined in \cite{ag12}. Our focus is more on the support subgraph rather than the ideal $I_{_M}$ generated by $V\setminus V_{_M}= \{ v \in V \>  \vert \> Mv=0 \}$, so we work with cohereditary and colorful instead of hereditary and $\Pi$-saturated.


\begin{lemma} \label{support}
 The following are equivalent for a subgraph $\Lambda$ of a (finitely separated) digraph $\Gamma :$\\
(i) $\Lambda =\Gamma_{_M}$, the support subgraph of a unital $L(\Gamma)$-module $M$.\\
(ii) $\Lambda$ is a full, cohereditary and colorful subgraph. \\
(iii) If $\Lambda=(V',E')$ then
\[ \theta(v)=
\left\{
\begin{array}{cc}
 v &   \> \> v \in V'  \\
  0 &   \> \>  v\notin V' 
\end{array}
\right. \> ; \quad \theta(e)=
\left\{
\begin{array}{cc}
 e &   \> \> e \in E'  \\
  0 &   \> \>  e\notin E'
  \end{array}
\right.    \> ; \quad \theta(e^*)=
\left\{
\begin{array}{cc}
 e^* &  \>  \> e \in E'  \\
  0 \> \> &  \> \> e \notin E'
\end{array} \right.
\]
defines an onto algebra homomorphism $\theta: L(\Gamma) \longrightarrow L(\Lambda)$.
\end{lemma}
\begin{proof}
($i$) $\Rightarrow$ ($ ii$): $\Gamma_{_M}$ is a induced subgraph hence full. If $te $ is in $\Gamma_{_M}$ then $0 \neq Mte=Me^*e \subseteq Me=M(se)e$, so $Mse \neq 0$, that is, $se $ is also in $\Gamma_M$ and $\Gamma_M$ is cohereditary.
 If $sX$ is in $\Gamma_{_M}$ then $0 \neq MsX \cong \sum_{e \in X} Mte$ implies that there is an $e$ in $X$ with $Mte \neq 0$. Thus $e$ is in $\Gamma_{_M}$  and $\Gamma_{_M}$ is colorful. \\

($ii$) $\Rightarrow$ ($ iii$): We need to check that $\theta$ preserves the defining relations of $L(\Gamma)$. No hypothesis is necessary to see that the path algebra  relations are satisfied.  For $e , f$ in X if $e\neq f$ then $e^*f=0$ in $L(\Lambda)$ as well as in $L(\Gamma)$. To see that $e^*e=te$ is preserved we need $\Lambda$ to be cohereditary and full. (If $e \in E'$ then $te\in V' $ and $e^*e=te$ holds in $L(\Lambda)$ also. If $te \in V'$ then $se \in V'$ since $\Lambda$ is cohereditary. So $e \in E'$ since $\Lambda$ is full. Again $e^*e=te$ holds in $L(\Lambda)$. Otherwise $e^*e=0=te$ in $L(\Lambda)$.) 
Finally, if $sX \in V'$ then $X\cap E' \neq \emptyset$ since $\Lambda$ is colorful. The image of $sX=\sum_{e \in X} ee^*$ under $\theta$ is $sX=\sum_{e \in E' \cap X} ee^*$, a defining relation of $L(\Lambda)$.\\

 ($iii$) $\Rightarrow$ ($ i$): Given a subgraph $\Lambda=(V',E')$ so that $\theta :L(\Gamma) \longrightarrow L(\Lambda) $ defines an algebra epimorphism, let $M:= L(\Lambda) \cong L(\Gamma)/ Ker \theta$. Now $v \in V'$ if and only if $\theta (v) \neq 0 $ and $Mv=L(\Lambda)v \neq 0$.
  Hence the vertex set of $\Gamma_{_M}$ is $V'$. But $\Lambda$ is full (if $te \in V' $ then $ 0 \neq \theta (te)= \theta (e^*) \theta(e)$, so $e \in E'$) hence $\Gamma_{_M}= \Lambda$. 
\end{proof}

  \begin{proposition} \label{rest.}
 If $M$ is a unital $L(\Gamma)$-module then $M$ also has the structure of a unital $L(\Gamma_{_M})$-module (where $\Gamma_{_M}$ is the support subgraph of $M$) inducing the $L(\Gamma)$ structure via the epimorphism $\theta :L(\Gamma) \longrightarrow L(\Gamma_{_M})$. Hence $Ker \theta  \subseteq AnnM$.
 \end{proposition}
 \begin{proof}
Let $\rho_{_M}$ be the quiver representation of $\Gamma$ corresponding to $M$ (as in Proposition \ref{teorem}). The restriction of $\rho_{_M}$ to $\Gamma_{_M}$ satisfies (I) because for all $X \in \Pi_{_M}$

$$ \rho_{_M} \vert_{_{\Gamma_M}} (sX)=MsX  \stackrel{\cong \quad }{\longrightarrow} \bigoplus_{e \in X } Mte \cong \bigoplus_{e \in X \cap E_M} Mte =\bigoplus_{e \in X \cap E_M} \rho_{_M} \vert_{_{\Gamma_M}}(te)$$ 
(since $Mte=0$ for $e \in X \setminus E_{_M}$ ). Let $M'$ be the unital $L(\Gamma_{_M})$-module coresponding to $\rho_{_M} \vert_{_{\Gamma_M}}$. Now $M'$ is also an $L(\Gamma)$-module via $\theta : L(\Gamma) \longrightarrow L(\Gamma_{_M})$.   As vector spaces $M'=\bigoplus_{v \in V_M} Mv  \cong \bigoplus_{v \in V} Mv=M$ by Lemma \ref{lemma} and since $Mv=0 $ for $v \in V\setminus V_{_M}$.
We can define an $L(\Gamma_{_M})$-module structure on $M$ via this isomorphism. But the action of the generatings $v\in V \> , \> \>e \in E ,\>   e^* \in E^*$ on $M$ and $M'$ is compatible with this isomorphism, so $M\cong M'$ as an $L(\Gamma)$-modules. Thus the $L(\Gamma)$-module structure of $M$ is induced from the $L(\Gamma_{_M})$-module structure via $\theta$.  
 \end{proof}\\
 
As an $L(\Gamma_M)$-module, $M$ has full support, that is, $Mv \neq 0$ for all $v \in V_M$.
 
 \begin{remark} \label{remark}
If  $I_{_M}$ is the kernel of  $\theta :L(\Gamma) \longrightarrow L(\Gamma_{_M})$ then $I_{_M}$ is generated by $V\setminus V_{_M}= \{ v \in V  \mid Mv=0 \}$.  
 \end{remark}

 \begin{proof} Let $J$ be the ideal generated by $V\setminus V_{_M}$. Clearly $V\setminus V_{_M} \subseteq I_{_M}$, hence we have the projection from $L(\Gamma)/J$ to $L(\Gamma)/I_{_M} \cong L(\Gamma_{_M})$. Conversely, let  $\varphi: L(\Gamma_{_M}) \longrightarrow L(\Gamma)/J$ defined by: $\varphi (v) =v+J \> ,\> \varphi (e) =e+ J \> ,\> \varphi (e^*) =e^* +J$.  The defining relations of $L(\Gamma_{_M})$, except for (SCK2), are trivially satisfied. If  $X\cap E_{_M} \neq \emptyset $ for $X \in \Pi$ then $\sum_{e \in X\cap E_{_M}} ee^*+ J=\sum_{e \in X} ee^*+ J  $ because $sX \in V_{_M}$ so $e \in X \setminus  E_{_M}$ if and only if $te \notin V_{_M}$.  But $\varphi$ is the inverse of the projection $L(\Gamma)/J \longrightarrow  L(\Gamma)/ I_{_M}$ thus $I_{_M}=J$.
\end{proof}

\begin{definition} \label{monoid}
The (additive) monoid $S(\Gamma)$ of the finitely separated digraph $\Gamma$ is generated by $V$ subject to the relations:\\

  $$sX =\sum_{e\in X} te  \quad \textit{ for all } X \textit{ in } \Pi .$$ 
	
\end{definition}


Hence, dimension functions of $\Gamma$ correspond exactly to monoid homomorphisms from $S(\Gamma)$ to $\mathbb{N}$ (natural numbers under addition). \\

$S(\Gamma)$  is isomorphic to the monoid $\mathcal{V} (L(\Gamma))$ of nonstable K-Theory of $L(\Gamma)$, (that is, isomorphism classes of finitely generated projective $L(\Gamma)$-module under direct sum). The generator $v$ of $S(\Gamma)$ corresponds to the (right) projective $L(\Gamma)$-module $v L(\Gamma)$ \cite[Theorem 3.5]{amp07}, \cite[Section 4]{ag12}, based on \cite{ber74}. The corresponding relations among the isomorphism classes of the cyclic projective modules $vL(\Gamma)$ was shown to hold in the proof of Proposition \ref{loc}. We can now reinterpret  the existence of a nonzero finite dimensional representation in terms of the nonstable $K$-theory of $L(\Gamma, \Pi)$.
 
 \begin{theorem}
 $L(\Gamma, \Pi)$ has a nonzero finite dimensional representation if and only if $\Gamma$ has  a finite, full, cohereditary and colorful subgraph $\Lambda$ with a nonzero monoid homomorphism from $\mathcal{V} (L(\Lambda))$ to $\mathbb{N}$.
 \end{theorem}
 \begin{proof}
 $L(\Gamma, \Pi)$ has a nonzero finite dimensional representation if and only if $\Gamma$ has  a nonzero dimension function of finite support by Corollary \ref{cor}. The support of this dimension function defines a finite, full, cohereditary and colorful subgraph $\Lambda$ and its restriction gives 
 a nonzero dimension function on $\Lambda$ (hence, a nonzero monoid homomorphism from $\mathcal{V} (L(\Lambda))$ to $\mathbb{N}$).\\
 
 Conversely, a nonzero dimension function on $\Lambda$ can be extended by 0 to a dimension function on $\Gamma$ since $\Lambda$ is 
 cohereditary and colorful. This gives a nonzero dimension function of finite support on $\Gamma$.  
  \end{proof}

\section{Leavitt path algebras of non-separated digraphs}

Recall that  when $\Pi=\left\{ s^{-1}(v) \mid v \in V,    s^{-1}(v)\neq \emptyset \right\}$, we say that $\Gamma$ is not separated. The \textit{Leavitt path algebra } $L(\Gamma,\Pi)$ and the \textit{Cohn path algebra } $ C(\Gamma, \Pi)$ are denoted by $L(\Gamma)$ and $C(\Gamma)$. The conditions (SCK1) and (SCK2) are denoted by (CK1) and (CK2) respectively \cite{aa05}, \cite{amp07}. Since $e^*f=e^*(se)(sf)f$ by ($E$), using ($V$) the relation (SCK1) is shortened to: (CK1) $e^*f =\delta_{e,f} te$ for all $e,f \in E$. \\



For any arrow $e$ in $E$ we have $e^*e=te$ by (CK1). Consequently $p^*p=t(p)$ for any path $p$ of $\Gamma$. Hence for any two paths $p$ and $q$ of $\Gamma$ if $q=pr$ then $p^*q=p^*pr=r$, if $p=qr$ then $p^*q=(q^*p)^*=r^*$. 
Using (CK1) we see that $p^*q=0$ unless the path $q$ is an initial segment of the path $p$  ($p=qr$) or $p$ is an initial segment of $q$ ($q=pr$). Thus the Cohn path algebra $C(\Gamma)$ and the Leavitt path algebra $L(\Gamma)$ are spanned by $\left\{pq^*\right\}$ where $p$ and $q$ are paths of $\Gamma$ with $tp=tq$. In fact this is a basis for $C(\Gamma)$
which can be shown by defining an epimorphism from $C(\Gamma)$ to a reduced semigroup algebra $\mathbb{F} S/ \mathbb{F}\{0\}$ where $S=\{pq^* \vert p, q \in Path(\Gamma)\>, \> tp=tq \} \sqcup \{0\}$ with the multiplication of $S$ defined formally as above. In $L(\Gamma)$ however if $E\neq \emptyset$ then $\{pq^* : tp=tq \}$ is linearly dependent because of (CK2).\\



\begin{lemma} \label{tensorCohn}Let $\Gamma$ be a row-finite digraph and $M$ an $\mathbb{F}\Gamma$-module. Then\\
(i) $C(\Gamma)= \oplus_{q \in Path (\Gamma)} \mathbb{F}\Gamma q^*$.\\
(ii) $\mathbb{F}\Gamma q^*\cong \mathbb{F}\Gamma tq$ as  left $ \mathbb{F}\Gamma$-modules.\\
(iii) $M \otimes_{\mathbb{F}\Gamma}  \mathbb{F}\Gamma q^*\cong M \otimes_{\mathbb{F}\Gamma}  \mathbb{F}\Gamma tq \cong Mtq$ as vector spaces.\\
(iv) $M \otimes_{\mathbb{F}\Gamma} C(\Gamma)
= M \otimes_{\mathbb{F}\Gamma} (\oplus_{q \in Path (\Gamma)} \mathbb{F}\Gamma q^* )\cong \oplus_{q \in Path(\Gamma)} Mtq$.
\end{lemma}

\begin{proof} (i) Since an $\mathbb{F}$-basis for $C(\Gamma)$ is $\{pq^* \vert p,q \in Path(\Gamma) \>, \>  tp=tq \}$, therefore $C(\Gamma)=\oplus \mathbb{F}\Gamma q^*$.

(ii) The isomorphism is given by $\cdot q$ and its inverse is $\cdot q^*$.

(iii) The first isomorphism is a consequence of (ii), and the second is $m \otimes \alpha \mapsto m\alpha$ with inverse $m \mapsto m \otimes tq$ for $m \in Mtq$.

(iv) This is a consequence of (i), (ii) and (iii).
\end{proof}

\begin{fact} \label{resolution}
The exact sequence $0 \rightarrow   I \rightarrow C(\Gamma) \rightarrow L(\Gamma) \rightarrow 0$ where $I$ is the ideal of $C(\Gamma)$ is a graded (with respect to any grading of $\mathbb{F}\Gamma$ with $V\sqcup E$ homogeneous), projective resolution of $L(\Gamma)$ as a left $\mathbb{F}\Gamma$-module.
\end{fact}
\begin{proof}
Since a vertex $v$ is an idempotent in $\mathbb{F}\Gamma$, hence the left $\mathbb{F}\Gamma$-module $\mathbb{F}\Gamma v$ is projective. 
By Lemma \ref{tensorCohn} (i) and (ii) $C(\Gamma)\cong  \oplus_{q \in Path (\Gamma)} \mathbb{F}\Gamma tq$, so $C(\Gamma)$ is a projective left
$\mathbb{F} \Gamma$-module. \\

When $\Gamma$ is finite $I$ is a projective $\mathbb{F} \Gamma$-module because $\mathbb{F} \Gamma$ is a hereditary ring. If $\Gamma$ is row-finite then  $\mathbb{F} \Gamma$ should still be hereditary, but this fact does not seem to be available in the literature.  
We will give another proof which also yields a more concrete discription of $I$ as an $\mathbb{F} \Gamma$-module.\\

All elements of $I$ are of the form $\sum_{i=1}^k \alpha_i (v_i-\sum_{se=v_i} ee^*) \beta_i$ where $\alpha_i$, $\beta_i$ are in $C(\Gamma)$ (because $C(\Gamma)$ has local units). However $f^*(v-\sum_{se=v} ee^*)=0=(v-\sum_{se=v} ee^*)f$ for all $f \in E$. Hence $I= \sum \mathbb{F}\Gamma (tq-\sum ee^*)q^*$, in fact the sum is direct: \\

 Let $\sum_{i=1}^n \alpha_i (tq_i -\sum ee^*) q_i^*=0 $ in I where the $q_i$s are distinct paths of $\Gamma$, $\alpha_i$ in $\mathbb{F}\Gamma tq_i$ and $l(q_1)\geq l(q_2) \geq \cdots \geq l(q_n)$. When $i>1$ either $q_i^* q_1=0$ or it is in $Path(\Gamma) \setminus V$. Hence $0=\sum_{i=1}^n \alpha_i (tq_i -\sum ee^*) q_i^*q_1=\alpha_1 (tq_1-\sum ee^*)$. The projection of  $\alpha_1 (tq_1-\sum ee^*)$ in  $C(\Gamma)$ to the summand $\mathbb{F}\Gamma tq_1$ is $\alpha_1$, thus $\alpha_1=0$. 
Similarly we get that $\alpha_i=0$ also for $i=2, \cdots , k$.  Therefore, $I=\oplus \mathbb{F}\Gamma (tq-\sum ee^*) q^*$.\\

Since $v-\sum ee^*$ is an idempotent in $C(\Gamma)$ and  $(\oplus_{se=v} \mathbb{F}\Gamma e^* )+  \mathbb{F}\Gamma v$ is a projective $\mathbb{F}\Gamma$-module, 
$\mathbb{F}\Gamma (v-\sum ee^*)=((\oplus_{se=v} \mathbb{F}\Gamma e^* ) \oplus \mathbb{F}\Gamma v)(v-\sum ee^*)$ is projective. Hence 
$\mathbb{F}\Gamma (tq-\sum ee^*)q^* \cong \mathbb{F}\Gamma (tq-\sum ee^*)$ is also projective. Consequently, $I$ is a projective left $\mathbb{F}\Gamma$-module. The inclusion of $I$ into $C(\Gamma)$ and the projection from $C(\Gamma)$ to $L(\Gamma)$ are graded homorphisms (with respect to any $G$-grading such that all elements of $V \sqcup E$ are homogeneous) so $0 \rightarrow   I \rightarrow C(\Gamma) \rightarrow L(\Gamma) \rightarrow 0$ is a graded projective resolution.
\end{proof}

\begin{example}\label{rose}
 When $\Gamma$ is the rose with $n$ petals $R_n$, 
let $M=\rho(v):= \mathbb{F}^{(\mathbb{N})}$ and let $\rho(e_i): \rho (v) \longrightarrow \rho(v)$ be 
given by $(a_0a_1a_2 \cdots) \rho(e_i)=(a_ia_{i+n}a_{i+2n} \cdots)$ where $e_i$, $i=0,\cdots , n-1$ are the loops of $\Gamma$. Condition (I) is satisfied, $\rho(e_i)$ are the downsampling maps and this representation of $L(R_n)$ in $ End \> \mathbb{F}^{(\mathbb{N})}$ gives the realization of $L(1,n)$ mentioned in the introduction.
 \end{example}
      
 \begin{example}\label{w}
Let $\Gamma$ be a row-finite digraph, $w$ a sink in $\Gamma$ and $P^w$ the set of all paths 
ending at $w$. We will define the $L(\Gamma)$-module $M^w$ via  the corresponding quiver representation $\rho^w$ where $\rho^w(v)$ is the vector space with basis $\{p \in P^w  \> \vert \> sp=v \}$ and 
$\rho^w(e)$ is the linear transformation defined as \[
p\rho^w(e) :=\left\{
\begin{array}{ccc}
  w & \textit{ if } p=e \qquad \qquad \\
  e_2e_3 \cdots e_n&  \textit{ if } p=ee_2e_3 \cdots e_n   \\
  0 &  \textit{otherwise.}  \qquad \quad      
\end{array}
\right.
\]
 Grouping the paths from $v$ to $w$ by their first arrow gives Condition $(I)$ since for every nonsink $v$ there is a bijection
 between the disjoint union of the given bases of $\rho^w (te)$ over $e \in s^{-1}(v)$ and the given basis of $\rho^w(v)$.
 \end{example}

 $P^w$ is an $\mathbb{F}$-basis of $M^w$ and (the proof of Theorem \ref{teorem} shows that) $pe^*=ep$. Hence the image of $pq^*$ in $End^{\mathbb{F}}(M^w)$ with $p$, $q$ in $P^w$ is the elementary matrix $E_{pq}$, thus $M^w$ is simple. The two-sided ideal $(w)$ of $L(\Gamma)$ is spanned by $\{pq^* \> \vert \> p, q \textit{ in } P^w \} $ which is linearly independent in $L(\Gamma)$ since the image set $\{ E_{pq} \}$ is linearly independent. Therefore $(w) \cong M_{n(w)}(\mathbb{F})$, the algebra of matrices indexed by $P^w$ with only finitely many nonzero entries where $n(w)$ is the number of paths ending at $w$.\\
  
    Mapping $p \in P^w$ to $p^*$ defines a homomorphism from $M^w$ to $wL(\Gamma)$ which is onto: 
    $\{pq^* \> \vert \> sp=w, \> tp=tq \> \}$ spans $wL(\Gamma)$, but $w$ is a sink so $p=w=tq$ and $\{q^* \> \vert \> tq=w \}$ spans $wL(\Gamma)$. Since $M^w$ is simple and $wL(\Gamma)$ is nonzero by Proposition \ref{1-1}, $M^w \cong wL(\Gamma)$ thus $M^w$ is projective. If $N$ is a finite direct sum of $\{M^w\}$ then $dim^{\mathbb{F}} (Nu)$ is the multiplicity of $M^u$ for any sink $u$. Hence there are no relations among the isomorphism classes of distinct $M^w$. In particular if $u\neq w$ are sinks then $M^u \ncong M^w$. \\

Defining the grade of $p$ in $P^w$ to be $-l(p)$ makes $M^w$ a graded $L(\Gamma)$-module. Then $E_{pq}$ is a graded homogeneous linear transformation of degree $l(p)-l(q)$. If for every vertex $v$ in $\Gamma$ there is a path from $v$ to a sink then we have a monomorphism from $L(\Gamma)$ to $\oplus End^\mathbb{F} (M^w)$ where the sum is over all sinks of $\Gamma$ (because this is a graded homomorphism whose kernel does not contain any vertex). \\

  When $\Gamma$ is finite and acyclic then $End^{\mathbb{F}}(M^w) \cong M_{n(w)}(\mathbb{F})$
  and the homomorphism from $L(\Gamma)$ to $\oplus M_{n(w)}(\mathbb{F})$  is onto since all the elementary matrices are in its image. Acyclicity  of $\Gamma$ and (CK2) yields that $\{ pq^*\>  \vert \> tp=tq=  \textit{sink } \} $ spans $L(\Gamma)$. Their images $\{ E_{pq} \} $ are linearly independent, so $L(\Gamma) \cong \oplus M_{n(w)}(\mathbb{F})$. Thus $M^w$ are the only simple modules of $L(\Gamma)$ and also $L(\Gamma)$ is finite dimensional. Conversely, if $L(\Gamma)$ is finite dimensional then $\Gamma $ has finitely many vertices and arrows as these are part of a basis of $\mathbb{F}\Gamma \subseteq L(\Gamma)$. If $\Gamma$ had a cycle $C$ then $C^k $ for $ k=1,2, \cdots$ would be linearly independent in $ \mathbb{F}\Gamma \subseteq L(\Gamma) $ contradicting fact that $L(\Gamma)$ is finite dimensional. Hence $L(\Gamma)$ is finite dimensional if and only if $\Gamma$ is finite and acyclic  \cite[Corollary 3.6]{aas07}. \\

The discussion above applies verbatim, proving a generalization to infinite digraphs: 

\begin{proposition} If $\Gamma$ is row-finite and has no infinite paths 
then $L(\Gamma) \cong \oplus M_{n(w)}(\mathbb{F})$ where the sum is over all sinks. $($Here
$M_{n(w)}(\mathbb{F})$ is the algebra of matrices indexed by $P^w$ with only finitely many nonzero entries.$)$ Also $M^w \cong w L(\Gamma)$ is a graded simple projective module for every sink $w$. 

\end{proposition}

\begin{example} \label{Toeplitz} 
An important instance of Example \ref{w} is the Toeplitz digraph:\\
$ \Gamma : \> \> 
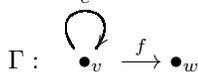
\xymatrix{  \> \> { \bullet}_{v} 
\ar@(ul,ur)^e } \xymatrix{ \stackrel{f}{\longrightarrow}  { \bullet}_{w} } \> $
\end{example}

 The basis given above of $M^w$ is $\{ w,\> f,\> ef,\> e^2f,\>  \cdots \}$ which can be identified with 
 $\mathbb{N}$ via the length function. Therefore $M^w \cong \mathbb{F}^{(\mathbb{N})}$ as vector spaces where $a_0w+a_1f+a_2ef+\cdots$ corresponds to $(a_0 \> a_1 \> a_2 \> \cdots )$, a finite $\mathbb{F}$-sequence. This representation of $L(\Gamma)$ also fits the framework of Proposition \ref{prop} with $d(v)=\infty$ and $d(w)=dim \>  M^ww=dim \> wL(\Gamma)w=1$. \\
 
    If $S$ and $T$ denote the images of $e+f$ and $e^*+f^*$ in $End \> \mathbb{F}^{(\mathbb{N})}$, respectively then $(a_0 \> a_1 \> a_2 \> \cdots )S= (a_1\> a_2 \> a_3 \> \cdots )$
   and $(a_0\>  a_1 \> a_2 \>\cdots )T= 
  (0\> a_0\>  a_1 \> a_2 \> \cdots )$.  We have:
 $v+w=1 \> $, $(e+f)(e^*+f^*) =ee^*+ff^*=v \>$, $(e+f)w=f \>$, $w(e^*+f^*)=f^*$ showing that $ e+f$
 and $ e^*+f^*$ generate $L(\Gamma)$.\\

Since $(e^*+f^*)(e+f)=e^*e+f^*f=v+w=1$, we have an epimorphism from the Jacobson \cite{jac50} algebra $\mathbb{F}\langle x,y \rangle :=\mathbb{F} \langle X,Y \rangle / (1-YX)$ to $L(\Gamma)$ sending $x$ to $e+f$ and $y$ to $e^*+f^*$. Composing this with the homomorphism from $L(\Gamma)$ to $End \> \mathbb{F}^{(\mathbb{N})}$ gives a monomorphism as $\{ x^my^n \> \vert \> m,\> n \in \mathbb{N} \}$ spans the Jacobson algebra and their images $\{ S^mT^n \> \vert \> m,\> n \in \mathbb{N} \}$ are linearly independent. Thus $L(\Gamma)$ is isomorphic to the Jacobson algebra and also the subalgebra of $End \> \mathbb{F}^{(\mathbb{N})}$ generated by $S$ and $T$. \\

$L(\Gamma)/(w) \cong \mathbb{F}[x,x^{-1}]$ since $w \leftrightarrow 1-xy$ in the isomorphism between $L(\Gamma)$ and  $\mathbb{F}\langle x,y \rangle$ above. 
The short exact sequence $M_\infty (\mathbb{F})\cong (w)  \hookrightarrow L(\Gamma) \twoheadrightarrow \mathbb{F}[x,x^{-1}]$ does not split \cite[Theorem 2]{aajz13}: If it were split then there would be a subalgebra $A$ of $End \> \mathbb{F}^{(\mathbb{N})}$ generated by $S+\alpha$ and $T+\beta$ isomorphic to $\mathbb{F}[x,x^{-1}]$ with $ x \leftrightarrow S+\alpha$ and $x^{-1}  \leftrightarrow T+\beta$, for some  $\alpha$ and $\beta$ with finite dimensional images. Considering $\mathbb{F}^{(\mathbb{N})}$ as a right $A\cong \mathbb{F}[x,x^{-1}]$-module we see that $S+\alpha$ and $T+\beta$ are inverses of each other. There is a $k$ with $\mathbb{F}^{(\mathbb{N})} \alpha \subseteq \mathbb{F}^k :=\{ a_0 a_1\cdots \> \vert \> a_n=0 \textit{ for } n\geq k \}$ and so $\mathbb{F}^{k+1}(S+\alpha) \subseteq \mathbb{F}^k$ because $\mathbb{F}^{k+1} S =\mathbb{F}^k$.
Thus $S+\alpha$ has a nontrivial kernel, contradicting that $S+\alpha$ is invertible.\\

 The short exact sequence $M_\infty (\mathbb{F}) \cong (w)  \hookrightarrow L(\Gamma) \twoheadrightarrow \mathbb{F}[x,x^{-1}]$ does not split as $L(\Gamma)$-modules either since $M_\infty (\mathbb{F})$ is not finitely generated. However, the inclusion $wL(\Gamma) \hookrightarrow L(\Gamma)$ does split since $v+w=1$, hence $vL(\Gamma) \oplus wL(\Gamma)=L(\Gamma)$. By Proposition \ref{loc}, $vL(\Gamma) \cong vL(\Gamma) \oplus wL(\Gamma) $ thus $L(\Gamma) \cong wL(\Gamma) \oplus L(\Gamma) \cong (wL(\Gamma))^n \oplus L(\Gamma)$ for any $n \in \mathbb{N}$. Therefore the category of finitely generated $L(\Gamma)$-modules does not have Krull-Schmidt because $wL(\Gamma)$ is simple by Example \ref{w}, hence indecomposable.  
(More generally, if $\Gamma$ is a finite digraph containing a cycle and a path from this cycle to a sink then the category of finitely generated representations of $L(\Gamma)$ does not have Krull-Schmidt.)\\
   
  When $\mathbb{F}= \mathbb{C}$ we can replace the vector space of finite $\mathbb{C}$-sequences   $\mathbb{C}^{(\mathbb{N})}$ with the Hilbert space $\mathit{l^2}$ of square summable sequences. Then $S$ and $T=S^*$ above are bounded operators (of norm 1) and 
  the closure with respect to the operator norm of the $*$-subalgebra $L(\Gamma)$ in the $C^*$-algebra of bounded linear operators $B(\mathit{l^2})$ generated by $S$ and $S^*$ is the classical Toeplitz algebra.

  \begin{example} $($Chen modules \cite{che15} $)$
  Let $\Gamma$ be a row-finite digraph, $\alpha=e_1e_2e_3 \cdots$ an infinite path and $[\alpha]$ the set of infinite paths $\beta =f_1f_2 f_3 \cdots $ having the same tail as $\alpha$ (that is, $f_{m+k}=e_{n+k}$ for all $k$ in $\mathbb{N}$, for some $m$ and $n$). 
  We will define the $L(\Gamma)$-module $M^\alpha$ via the quiver representation $\rho^\alpha$ as follows:
   $\rho^\alpha(v)$ is the $\mathbb{F}$-vector space with basis $\{\beta \in [\alpha] \> \vert \> s\beta=v \}$ and 
   $\rho^\alpha(e)$ is the linear transformation defined as  
  
  \[
\beta \rho^\alpha(e) :=\left\{
\begin{array}{cc}
  e_2e_3 \cdots &  \textit{ if } \beta=ee_2e_3 \cdots    \\
  0 &  \textit{ otherwise.}       
\end{array}
\right.
\]
 Grouping the paths starting at $v$ in $[\alpha]$ by their first arrow gives Condition (I). Thus $[\alpha]$ is a basis for $M^{\alpha}$ and the proof of Theorem \ref{teorem} shows that $\beta e^*=e\beta$ when $te=s\beta$ and 0 otherwise. Also $\beta pq^*=q \gamma $ if $\beta =p\gamma$ and 0 otherwise, implying that $M^\alpha$ is simple. 
\end{example}     

   There are two types of $M^\alpha$ depending on whether $\alpha$ is (eventually) periodic (that is, we can find $m$, $n$ so that $e_{k+n}=e_k$ for $k>m$)  or not. 
  When $\alpha$ is not periodic, $M^{\alpha}$ is a graded $L(\Gamma)$-module: the degree of $f_1f_2 f_3\cdots $ is $m-n$ where $m$ and $n$ are the positive integers satisfying $f_{m+k}=e_{n+k}$ for all $k \in \mathbb{N}$. If  $\alpha$ is periodic, picking the $n$, $m$ above smallest possible with $C:= e_{m+1} e_{m+2} \cdots e_{m+n}$, we get a bijection between the set of paths $P^C:= \{p \> \vert \> tp=sC, \> p  \textit{ does not end with } C \} $ and $[\alpha]$ given by $p \leftrightarrow p C^\infty$. Via this identification the image of $pC^l q^*$ with $l \in \mathbb{N} $ and $p,q \in P^C$ in $End^{\mathbb{F}} (M^\alpha ) $ is $E_{pq}$. \\

When $\Gamma$ is $R_n$, the rose with $n$ petals and $\alpha= e_0e_0e_0\cdots$ then $M^\alpha$ above is (isomorphic to) the module of Example \ref{rose}.

%

 




%

\section{Finite dimensional quotients of the Leavitt path algebra of a row-finite digraph}

In this section, $\Gamma$ will be a non-separated digraph, that is, $\Pi := \{ s^{-1}(v) \mid s^{-1}(v) \textit{ \em{nonempty} } \}$. In the non-separated context \textit{finitely separated} means row-finite. Recall that an $L(\Gamma)$-module is of \textit{finitary} if $dim^{\mathbb{F}}(Mv)< \infty $ for all $v \in V$. 
When $V$ is finite, finitary is the same as finite dimensional since $M =\oplus_{v \in V}  Mv$ by Lemma \ref{lemma}.
\begin{lemma} \label{exits}
If an $L(\Gamma)$-module $M$ is finitary then the cycles of its support subgraph $\Gamma_{_M}$ have no exits. 

\end{lemma}
\begin{proof}
If $v_1,..., v_n, v_{n+1}=v_1$ are consecutive vertices in a cycle of $\Gamma_{_M}$ then $dim(Mv_1) \geq dim(Mv_2) \geq \cdots \geq dim(Mv_n)\geq dim( Mv_1)$ by Lemma \ref{lemma}. Hence $dim(Mv_k) =dim(Mv_{k+1})$ for $k=1, \cdots, n$. It follows from Lemma \ref{lemma} again that $Mte=0$ for $e\in s^{-1}(v_k)$
unless $te=v_{k+1}$. Thus cycles of $\Gamma_{_M}$ have no exits. 
\end{proof}\\

Next we characterize all possible finite dimensional quotients of the Leavitt path algebra of a row-finite digraph as direct sums of matrix algebras over finite dimensional cyclic algebras.

\begin{theorem} \label{teo1}
 If $A$ is a finite dimensional quotient of the Leavitt path algebra $L(\Gamma)$ of a row-finite digraph $\Gamma$ then $A\cong  \bigoplus\limits_{k=1}^{m} M_{n_k} (B_k)$, where each $n_k$ is a positive integer, $B_k = {\mathbb{F}\left[x\right]}/\left( P_k(x) \right)$ with $P_k(x)$ non-constant and $P_k(0)=1, \quad k=1,2,...,m.$
\end{theorem}
\begin{proof} 
If $A=L(\Gamma)/I$ is a finite dimensional quotient of $L(\Gamma)$ then $A$ is a unital $L(\Gamma)$-module. Its support subgraph $\Gamma_{A}$ is finite by Lemma \ref{lemma} and the cycles of $\Gamma_{A}$ have no exits by Lemma \ref{exits}. Let $I_{A}$ be the ideal generated by $V\setminus V_{A}= \{ v \in V \mid L(\Gamma)v =Iv \}=V\cap I $ as in Remark \ref{remark}. We have a homomorphism from $L(\Gamma_{A} ) \cong L(\Gamma)/I_{A}$ onto $L(\Gamma)/I=A$ (since $I_{A}$ is generated by $I\cap V$). So we may replace $\Gamma$ with $\Gamma_{A}$, a finite digraph whose cycles have no exits.\\

$L(\Gamma_{A})$ is isomorphic to a direct sum of matrix algebras over $\mathbb{F}$ and/or $~\mathbb{F}\left[x,x^{-1}\right]$ by \cite[Theorem 3.8 and 3.10]{aas08} 
(the number of summands of the form $M_n (\mathbb{F})$ is the number of sinks in $\Gamma_{A}$ and the number of summands of the form $M_n (\mathbb{F}\left[x,x^{-1}\right])$ is the number of cycles in $\Gamma_{A}$) and $A\cong {L(\Gamma_{A})}/J$. From now on we will identify $L(\Gamma_{A})$ with this direct sum of matrix algebras. \\

Let $\pi^{k}$ be the projection from $L(\Gamma_{A})$ to the $k$ th factor  $M_{n_k}(\mathbb{F})$ or $M_{n_k}(\mathbb{F}[x,x^{-1}])$ and $\pi_{ij}^k$ be $\pi^k$ composed with the projection to the $ij$-th entry. Note that multiplying on the left by $E_{li}$ and on the right by $E_{jm}$ in the $k$-th coordinate moves the $ij$-th entry to the $lm$-th entry, hence $\pi_{ij}^k(J)$ is independent of $ij$ (since $J$ is an ideal). If $J_k:=\pi_{ij}^k(J)$ then $J_k$ is an ideal of $\mathbb{F}$ or $\mathbb{F}\left[x,x^{-1}\right]$ and $J = \oplus M_{n_k}(J_k)$: 
We have $J \subseteq \oplus M_{n_k}(J_k)$ by the definition of the $J_k$.  To see the converse note that $\oplus M_{n_k}(J_k)$ is generated by $E_{ij}^k\alpha $ with $\alpha \in J_k$ where $E_{ij}^k$ denotes the element with $E_{ij}$ in the $k$-th coordinate and $0$ all the others. If $\alpha =\pi_{ij}^k (\beta)$ with $\beta \in J$ then $E_{ii}^k \beta E_{jj}^k =  E_{ij}^k\alpha$, thus $E_{ij}^k\alpha \in J$ and $\oplus M_{n_k}(J_k) \subseteq J$.\\

If $J_k \lhd \mathbb{F}$ then either $J_k=\mathbb{F}$, in which case the corresponding summand does not appear in $A$, or $J_k=0$ and the summand $M_{n_k}(\mathbb{F}) \cong M_{n_k} (B_k)$ where $B_k=F[x]/(x-1)$. If  $J_k \lhd F[x,x^{-1}]$ then either $J_k=\mathbb{F} [x,x^{-1}]$ so the corresponding summand does not appear in A, or $J_k= (P_k(x))$ and we may assume that $P_k$  is non-constant, $P_k \in F[x]$ and $P_k (0)=1$ (multiplying with a power of $x$ if necessary). Now $B_k:=F[x]/(P_k(x)) \cong F[x,x^{-1}]/ (P_k(x)) $ and $A \cong \oplus M_{n_k} ( \mathbb{F} [x,x^{-1}] ) / M_{n_k}(J_k) \cong \oplus M_{n_k} (B_k) $.
\end{proof}\\

As mentioned in the final paragraph of section 2 above, the graded ideals of $L(\Gamma)$ when $\Gamma$ is a row-finite digraph are in 1-1 correspondence with hereditary saturated subsets of vertices \cite[Theorem 5.3]{amp07}. This correspondence is given by sending a graded ideal $I$ to $V \cap I$ and its inverse sends a hereditary saturated subset $S$ of $V$ to $(S)$, the (graded) ideal generated by $S$. In particular, if $I$ is a graded ideal then $I= ( I\cap V)$. Consequently;\\
\begin{fact} \label{graded}
A graded ideal $I$ of $L(\Gamma)$ is nonzero if and only if $I\cap V\neq \emptyset$. Hence, a graded homomorphism $\varphi$ from $L(\Gamma)$ is one-to-one if and only if $\varphi(v)\neq 0$ for all $v \in V$. 
\end{fact}

If a grading is a refinement of another then an ideal graded with respect to the finer grading is clearly also graded with respect to the other. The converse holds for the universal grading and standard $\mathbb{Z}$-grading, i.e., an ideal is graded with respect to one if and only if with respect to the other (since $I$ is generated by $I\cap V$ which consists of homogeneous elements).\\

A subset $S$ of $V$ is hereditary if and only if the induced subgraph on its complement $V \setminus S$ is cohereditary. Also $S$ is saturated if and only if the induced subgraph on $V\setminus S$ is colorful for a non-separated digraph. Hence we can add a fourth equivalent condition to Lemma \ref{support} (for $\Gamma$ a row-finite digraph): There is a 1-1 correspondence between graded ideals $\{I\}$ of $L(\Gamma)$ and support subgraphs $\{ \Gamma_{_M}= (V_{_M},E_{_M} )\}$ given by $I=(V\setminus V_{_M}) $. Thus, for any ideal $I$ of $L(\Gamma)$ the unique maximal graded ideal $J$ contained in $I$ is $(V\setminus V_{_{L(\Gamma)/I}})$. Moreover the modules $L(\Gamma)/I$ and $L(\Gamma)/J$ have the same support subgraph.

\begin{corollary}
If $I$ is a graded ideal of $L(\Gamma)$ with $dim(L(\Gamma)/I)$ finite then $A:=L(\Gamma)/I$ is isomorphic to a direct sum of matrix algebras over $\mathbb{F}$.

\end{corollary}

\begin{proof}
When $I$ is graded $I=(V \setminus   V_{A})$ as explained above. There are no cycles in $\Gamma_{A}$ because $L(\Gamma_{A}) \cong L(\Gamma)/I$ is finite dimensional. Hence the only summands of $L(\Gamma)/I$ are matrix algebras over $\mathbb{F}$. 
\end{proof}\\



 
In order to state a necessary and sufficient criterion (in terms of the digraph $\Gamma$) for the existence of a nonzero finite dimensional quotient of $L(\Gamma)$ we need a few  definitions. 
We say $v$ \textit{connects to} $w$, denoted $v \leadsto w$, if there is a path $p$ in $\Gamma$ such that $sp=v$ and $tp=w$. This defines a preorder (reflexive and transitive relation) on the vertices of $\Gamma$. If $v$ and $w$ are on a cycle then $v \leadsto w$ and $w \leadsto v$. Let $U$ be the set of sinks and cycles of $\Gamma$. There is an induced preorder on $U$, also denoted by $\leadsto$. (This is a partial order on $U$ if and only if the cycles of $\Gamma$ are disjoint.) A sink or a cycle $u \in U$ is \textit{maximal} if $u' \leadsto u$ only if $u' =u$. \\

 The \textbf{\textit{predecessors}} of $v$ in $V$ is $V_{\leadsto v} := \{ w \in V \mid w \leadsto v\} $. If $u$ and $w$ are two vertices on a cycle $C$ then they have the same predecessors, so $V_{\leadsto C} $ is well-defined. Let $\Gamma_{\leadsto v}$ be the induced subgraph on $V_{\leadsto v}$ .
 
 \begin{theorem} \label{maximal}
Let $\Gamma$ be a row-finite digraph. $L(\Gamma)$ has a nonzero finite dimensional module (equivalently a nonzero finite dimensional quotient) if and only if $\Gamma$ has a maximal sink or cycle with finitely many predecessors.
\end{theorem}
\begin{proof}
Having a nonzero finite dimensional quotient is equivalent to having a nonzero finite dimensional module: Any quotient is also a module and conversely if $M$ is a nonzero finite dimensional $L(\Gamma)$-module then there is a nonzero homomorphism from $L(\Gamma)$ into $End(M)$ 
 whose image is finite dimensional.\\

If there is a nonzero finite dimensional quotient $L(\Gamma)/I$ let $M=L(\Gamma)/I$ and $\Lambda:=\Gamma_{_M}$,  its support subgraph. $\Lambda$ is a finite digraph (Lemma \ref{lemma}) whose cycles have no exits (Lemma \ref{exits}). If $\Lambda$ has a sink $w$ then $w$ is also a sink in $\Gamma$ because $\Lambda$ is colorful by Lemma \ref{support}. There is no path from any cycle in $\Gamma$ to $w$ since this cycle would be a cycle with an exit in $\Lambda$ (as $\Lambda$ is cohereditary). Hence $w$ is a maximal sink. If there is no sink then the finite digraph $\Lambda$ must have a cycle. This cycle has to be maximal, as above, otherwise $\Lambda$ would have a cycle with an exit. The predecessors of this maximal sink or cycle is contained in $\Lambda$ so it is finite.\\

Conversely, if $\Gamma$ has a maximal sink or cycle with finitely many predecessors  then the induced subgraph $\Lambda$ on this finite set $W$ of predecessors is full, cohereditary and colorful. So $L(\Lambda)$ is a quotient of $L(\Gamma)$ by Lemma \ref{support}. Moreover, there is at most one cycle in $\Lambda$ which has no exits. Thus $L(\Lambda) \cong M_n (\mathbb{F}) $ if there is no cycle or $L(\Lambda) \cong M_n (\mathbb{F}[x,x^{-1}])$ when there is a cycle (as in the proof of Theorem \ref{teo1} above). In both cases the finite dimensional algebra $M_n(\mathbb{F})$ can be realized as a quotient of $L(\Lambda)$ hence also of $L(\Gamma)$.
\end{proof}\\

\begin{remark}
Theorems 6.2 and 6.5 with some additional work  can yield a description of all finite dimensional indecomposable modules of Leavitt path algebras of (non-separated) row-finite digraphs. However, these are already classified in Theorems 4.4 and 4.9 of \cite{koc1} using a different approach. They are as follows.\\

There are two kinds of finite dimensional indecomposable unital $L(\Gamma)$-modules: If $M$ is of the first kind then $M$ is completely determined by a maximal sink $v$ with finitely many predecessors. The subspace $Mu$ for any vertex $u$ has dimension equal to the number of paths from $u$ to $v$ (hence the support of $M$ is the predecessors of $v$). The linear transformation given by an arrow $e$ is essentially a projection (corresponding to the injection from the set of paths $te$ to $v$ to the set of paths $se$ to $v$ by appending $e$ to the beginning of each path). These indecomposables are simple. \\

An indecomposable $M$ of the second kind is determined by a maximal cycle $C$ with finitely many predecessors, an irreducible $f(x) \in \mathbb{F}[x]$ with $f(0)=1$ and a positive integer $n$. Now, $Mv\cong \mathbb{F}[x]/ (f(x)^n)$ for any vertex $v$ on $C$. The linear transformation given by one of the arrows on $C$ corresponds to multiplication by $x$, the remaining arrows on $C$ give the identity transformation (isomorphism type of $M$ is independent of which arrow corresponds to multiplication by $x$). For any vertex $u$, $Mu$ is isomorphic to a direct sum of copies of $Mv$ indexed by the paths from $u$ to $v$, which do not traverse $C$. The arrows that are not on the cycle $C$ give projections, as above. These indecomposables are simple if and only if $n=1$.
  \end{remark}
  
For a  ring $R$ with 1, the UGN (Unbounded Generating Number) property is: if $R^m \cong R^n \oplus P$ as $R$-modules then 
$m\geq n$. Equivalently, $R$ does not have UGN if and only if $R^{m+1}$ is a quotient of $R^m$ (up to isomorphism) for some $m$. 
We define the \textit{non-UGN type} of $R$ to be the smallest such $m$. For a non-UGN Leavitt path algebra we prove below Corollary \ref{ugn} that its type is always 1.  Corollary \ref{ugn} also provides a different proof of the characterization of the UGN property for Leavitt path algebras \cite[Theorem 3.16]{anp} \\

Clearly, UGN implies IBN. Also the existence of a nonzero finite dimensional quotient implies UGN (by a dimension count after tensoring with this quotient). 
The UGN property of $L(\Gamma)$  is characterized in terms of $\Gamma$ in \cite[Theorem 3.16]{anp} for a finite digraph $\Gamma$. Even though their characterization is expressed quite differently, it is not difficult to see that it is equivalent to the existence of a maximal sink or a maximal cycle. $L(\Gamma)$ has UGN if and only if it is algebraically amenable by \cite[Remark 3.17 ]{anp}. This is a consequence of Corollary 5.11 in \cite{allw18} where the concept of algebraic amenability was introduced.

\begin{corollary} \label{ugn}
Let $\Gamma$ be a finite digraph. Then $\Gamma$ has a maximal sink or a maximal cycle if and only if $L(\Gamma)$ has  UGN if and only if $L(\Gamma)$ is algebraically amenable if and only if $L(\Gamma)$ has a nonzero finite dimensional quotient if and only if $L(\Gamma)\oplus L(\Gamma)$ is not isomorphic to a quotient module of $L(\Gamma)$.
\end{corollary}
\begin{proof}
If $\Gamma$ has a maximal sink or a maximal cycle then $L(\Gamma)$ has a nonzero finite dimensional representation by Theorem \ref{maximal},
hence $L(\Gamma)$ has UGN as explained above.\\

When $\Gamma$ does not a maximal sink or a maximal cycle let $U$ be the set of vertices in $\Gamma$ lying on at least two cycles and let 
$P=\oplus_{u\in U} uL(\Gamma)$.  Since $L(\Gamma)=\oplus_{v \in V}vL(\Gamma)$ we see that $P$ is a quotient module
 of $L(\Gamma)$.\\
 
 Using the isomorphisms $uL(\Gamma) \cong \oplus_{se=u} teL(\Gamma)$ of Proposition \ref{loc} repeatedly, for each $u\in U$ we can express $uL(\Gamma)$ as a direct sum with at least two of the summands being $uL(\Gamma)$, since $u$ lies on multiple cycles. Hence $P\oplus P$ is quotient of $P$.\\

If $w$ is a sink then $w$ is a descendant of a cycle since there are no maximal sinks. If $v \in V$ is on a cycle which meets another cycle then 
 $v$ is a descendant of some $u \in U$. If $v$ is on a cycle which is disjoint from all other cycles then $v$ is a descendant of 
 a different cycle since there are no maximal cycles. Repeating this if necessary we get that $v$ is a descendant of some $u$ in $U$ because $\Gamma$ is finite. Thus all sinks and all vertices on a cycle are descendants of $U$.\\
 
 For each $v\in V$ all paths starting at $v$ eventually reach a sink or a cycle. 
 Thus using $vL(\Gamma) \cong  \oplus_{se=v} teL(\Gamma)$ repeatedly we can express $vL(\Gamma)$ as a direct sum with summands $wL(\Gamma)$ where each $w$ is a descendant of $U$. If $w$ is a descendant of $u$ in  $U$ then $wL(\Gamma)$ is a quotient of $uL(\Gamma)$ 
 as above. Hence $vL(\Gamma)$ and also $L(\Gamma)=\oplus vL(\Gamma)$ are quotients of $P^k$ for some $k$.\\
 
Now $P$ is a quotient of $L(\Gamma)$, also $P\oplus P$ hence $P^{2k}$ are quotients of $P$ and $L(\Gamma)\oplus L(\Gamma)$ is a quotient of 
$P^{2k}$. Therefore if $\Gamma$ has no maximal sinks or maximal cycles then $L(\Gamma)\oplus L(\Gamma)$ is a quotient of $L(\Gamma)$.
\end{proof}

\begin{remark}
For every positive integer $m$ there are infinitely many non isomorphic algebras of non-UGN type $m$, namely the Leavitt algebras $L(m,n)$ 
for all $n >m$. These algebras can be realized as the corner algebras $wLw$ of separated Leavitt path algebras $L=L(\Gamma_{m,n})$ where $\Gamma_{m,n}$ is the digraph with two vertices $u$, $w$ and $m+n$ arrows from $u$ to $w$ separated into a part of $m$ and another part of $n$
 arrows \cite[Proposition 2.12 (1)]{ag12}. Computing the non-stable $K$-theory monoid $\mathcal{V} (wLw)$ gives: 
 (i) Every finitely generated projective $wLw$ is free. 
 (ii) All isomorphisms between finitely generated projective $wLw$-modules are consequences of $(wLw)^m \cong (wLw)^n$. Consequently, 
 $(wLw)^{k+1}$ is a quotient of  $(wLw)^k$ if and only if $k\geq m$.
  \end{remark}

   \begin{corollary} \label{reversible}
 We have the following implications for the Leavitt path algebra of a finite digraph, neither  of which is reversible: $L(\Gamma)$ has finite Gelfand-Kirillov dimension implies that $L(\Gamma)$ has a nonzero finite dimensional quotient implies that $L(\Gamma)$ has IBN.
 \end{corollary}
 
 \begin{proof}
   If $L(\Gamma)$ has finite Gelfand-Kirillov dimension then the cycles in $\Gamma$ are disjoint \cite[Theorem 5]{aajz12}. Thus $\Gamma$ must have a maximal sink or a maximal cycle and by Theorem \ref{maximal}, $L(\Gamma)$ has a nonzero finite dimensional quotient. If $L(\Gamma)$ has a nonzero finite dimensional quotient then $L(\Gamma)$ has IBN (since finite dimensional unital algebras have IBN and if  a unital ring does not have IBN then neither does any nonzero homomorphic image of it).\\
     
    The examples below show that neither implication is reversible:\\
 
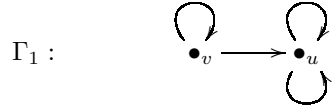
$$ \Gamma_1 : \quad  \xymatrix{  &\> \> { \bullet}_{v} \ar[r]
\ar@(ul,ur)   &{ \bullet}_{u} 
\ar@(ul,ur) \ar@(dl,dr) } 
 $$\\
 
There is no path to the loop at $v$ from any other cycle, hence $L(\Gamma_1)$ has a nonzero finite dimensional quotient. But the Gelfand-Kirillov dimension is infinite since the loops at $u$ are not disjoint.\\

To see that the second implication is not reversible consider the digraph $\Gamma_2$ below:

$$ \Gamma_2 : \quad  \xymatrix{  &\> \> { \bullet}_{v} 
\ar@(ul,ur) \ar@(dl,dr)  \ar@/^1pc/[r] \ar[r]  & {\bullet}_{u}  } $$ \\
$L(\Gamma_2)$ has no nonzero finite dimensional quotient (both cycles and the sink are reachable from another cycle). But $L(\Gamma_2)$ has IBN by the criterion of Kanuni-Özaydın \cite{ko}: The only relation we have is $v=2v+2u$, yielding $(1,2)$. Then $L(\Gamma_2)$ has IBN since $(1,1)$ is not in the $\mathbb{Q}$-span of $(1,2)$. Another way to see that $L(\Gamma_2)$ has IBN is to note that $L(\Gamma_2)$ is isomorphic to a Cohn path algebra (of the rose with 2 petals). Cohn path algebras have IBN \cite{ak}.
\end{proof}

    \section{$\mathfrak{M}_{\mathbb{F}\Gamma}$ versus $\mathfrak{M}_{L(\Gamma)}$  }
    
 $\mathfrak{M}_{L(\Gamma)}$ is a full subcategory of  $\mathfrak{M}_{\mathbb{F}\Gamma}$  by Proposition \ref{teorem} and $\mathfrak{M}_{L(\Gamma)}$ is a retract of  $\mathfrak{M}_{\mathbb{F}\Gamma}$  by Theorem \ref{forgetful}, that is,   
 the composition of the forgetful functor from $\mathfrak{M}_{L(\Gamma)}$ to $\mathfrak{M}_{\mathbb{F} \Gamma}$ with  $\underline{\>\>\> } \otimes_{\mathbb{F}\Gamma} L(\Gamma)$ from $\mathfrak{M}_{\mathbb{F} \Gamma}$ to $\mathfrak{M}_{L(\Gamma)}$
is naturally equivalent to the identity functor on $\mathfrak{M}_{L(\Gamma)}$. Thus,  $\mathfrak{M}_{L(\Gamma)}$ is also a quotient of $\mathfrak{M}_{\mathbb{F} \Gamma}$. In this section we identify explicitly the Serre subcategory of $\mathfrak{M}_{\mathbb{F} \Gamma}$ that we quotient out and 
along the way we realize the functor $\>\> \underline{\> \>\>} \otimes_{\mathbb{F}\Gamma} L(\Gamma)$ via a direct limit construction.\\

  We will need the following generalization of (CK1) and (CK2) which requires a definition: Let $E_n:=\{ p\in Path(\Gamma)\vert l(p)=n \textit{ or } l(p)< n \textit{ and } tp \textit{ is a sink }\} $ and $E_n^v := \{ p \in E_n \vert sp=v \}$ for $n \in \mathbb{N}$. In particular,  $E_0=V$ and $E_1=\mathcal{S} \sqcup E$ where $\mathcal{S}$ is the set of sinks in $\Gamma$. If $w$ is a sink then $E_n^w=\{w \}$ for all $n \in \mathbb{N}$.
  
 \begin{lemma} \label{CK+} For all $n \in \mathbb{N}$ and $v \in V$\\
 (i) $p^*q=\delta_{p,q}tp$ for all $p,\>q \in E_n$;\\
 (ii) $\{pp^* \vert p \in E_n \} $ is a set of orthogonal idempotents;\\
 (iii) $v=\sum pp^*$ where the sum is over $ p \in E_n^v$.
  \end{lemma}
 \begin{proof}
(i) $p^*q\neq 0$ if and only if $p$ is an initial segment of $q$ or $q$ is an initial segment of $p$. This is possible only if $p=q$ because either $l(p)=n=l(q)$ or $tp$ or $tq$ is a sink. Also $p^*p=tp$.\\
(ii) This follows directly from (i).\\
(iii) For $n=0$ this says $v=v$. For $n >0$ this follows from repeated applications of (CK2) until $l(p)=n$ or $tp$ is a sink. 
 \end{proof}\\

To understand the functor $\>\> \underline{\> \>\>} \otimes_{\mathbb{F}\Gamma} L(\Gamma) : \mathfrak{M}_{\mathbb{F}\Gamma} \longrightarrow \mathfrak{M}_{L(\Gamma)}$ better we will give an alternate model for $M \otimes L(\Gamma)$.
 If $M$ is an $\mathbb{F}\Gamma$-module then we define the $\mathbb{F}\Gamma$-modules $\sigma^k M$ for all $k \in \mathbb{N}$ as a quiver representation: $(\sigma^k M)v:= \bigoplus_{p \in E_k^v} Mtp$  for all $v\in V$. To define the linear transformation given by $e \in E$ from $(\sigma^k M)se = \bigoplus Mtp$ to $(\sigma^k M)te=\bigoplus Mtq$ 
we focus on a single block $Mtp \longrightarrow Mtq$. This is defined to be zero unless $p= ep'$ and $q=p'f$ with $f \in E_1^{tp}$, in which case it is 
right multiplication by $f$. Note that $\sigma^0 M=M$ and $\sigma^k (\sigma^l M)=\sigma^{k+l} M$ for $k,\> l $ in $\mathbb{N}$.  For an $\mathbb{F}\Gamma$-module $N$ we have a module homomorphism $\theta_N : N \longrightarrow  \sigma N$ given by $Nv \stackrel{(\cdot f)}{\longrightarrow} \oplus_{f \in E_1^v} Mtf=(\sigma N)v$. We get a directed system $M \stackrel{\theta_M}{\longrightarrow} \sigma M  \stackrel{\theta_{\sigma(M)}}{\longrightarrow} \sigma^2 M \longrightarrow \cdots $. We also have $\mathbb{F}\Gamma$-module homomorphisms
from $\sigma^k M$ to $M \otimes L(\Gamma)$ given by $m \mapsto m\otimes p^*$ for $m\in Mtp \subseteq \sigma^k M$ yielding a commutative triangle \[
\left.
\begin{array}{ccc}
 \sigma^kM  & \stackrel{\theta_{\sigma^kM}}{\longrightarrow}   &  \sigma^{k+1}M \\
 &   \searrow& \downarrow  \\
  &   &M\otimes L(\Gamma)   
\end{array}
\right.
\]
We get a homomorphism from the direct limit $colim \> \sigma^k M$ to $M\otimes L(\Gamma)$.\\

  \begin{theorem} \label{isomorphism}
 If $M$ is an $\mathbb{F}\Gamma$-module then $colim\> \sigma^kM$ is an $L(\Gamma)$-module naturally isomorphic to $M\otimes_{\mathbb{F}\Gamma} L(\Gamma)$.
  \end{theorem} 
 \begin{proof}
 The homomorphism from  $colim\> \sigma^kM$ to $M \otimes L(\Gamma)$ was defined above. In the opposite direction we want to show that the linear transformation from $M\otimes L(\Gamma)$ to $colim \sigma^kM$ sending $m \otimes p^*$ to $[m]$ where $m \in Mtp \subseteq \sigma^{l(p)} M$ is well-defined. Since $\{pq^* \> \vert \> p,q \in Path(\Gamma) \>, \>  tp=tq \> \}$ is a basis for $C(\Gamma)$ and $C(\Gamma)=\oplus_{v \in V} vC(\Gamma)$ we have $M\otimes C(\Gamma) \cong \oplus_{q \in Path(\Gamma)} Mtq$.  
Thus $m \otimes p^* \mapsto [m]$ is a well-defined linear transformation from $M \otimes C(\Gamma)$ to $colim \sigma^k M$. Since $m \otimes (v-\sum_{se=v} ee^*)q^* \mapsto 0$ for every non sink $v \in V$ the linear transformation 
 $M \otimes L(\Gamma) \longrightarrow colim \sigma^kM$ is defined. Since these homomorphisms are inverses of each other we have 
 $M \otimes_{\mathbb{F}\Gamma} L(\Gamma) \cong colim \sigma^kM$ and the isomorphism is natural, i.e., for any $\mathbb{F}\Gamma$-module homomorphism $f:M \longrightarrow N$ we get a commutative diagram \[
\left.
\begin{array}{ccc}
  colim \sigma^kM &  \stackrel{colim \sigma^k f}{\longrightarrow}  &  colim \sigma^kN \\
 \cong  \downarrow &   & \downarrow \cong   \\
  M\otimes L(\Gamma)&  \stackrel{f\otimes id_{L(\Gamma)}}{\longrightarrow}   &N\otimes L(\Gamma)   
\end{array}
\right.
\]
The $\mathbb{F}\Gamma$-module $colim \sigma^kM$ satisfies condition (I) of Proposition \ref{teorem} since it is isomorphic to 
the $L(\Gamma)$-module $M\otimes L(\Gamma)$. Thus $colim \sigma^kM$ and $M\otimes L(\Gamma)$ are isomorphic 
as $L(\Gamma)$-modules. 
\end{proof}\\

 When $\Gamma$ is finite the fact that $L(\Gamma)$ is a flat $\mathbb{F}\Gamma$-module is proven in \cite[Proposition 4.1]{ab10}.
 Below we give a different proof for a row-finite $\Gamma$ using Theorem \ref{isomorphism}.  This is not a consequence of $L(\Gamma)$
 being a localization of $\mathbb{F}\Gamma$ (Proposition \ref{loc}) since universal localizations are not necessarily even stably
 flat \cite[Lemma 1.4]{nrs04}.

\begin{lemma} \label{flat}
 $L(\Gamma)$ is a flat left $\mathbb{F}\Gamma$-module.
 \end{lemma}
\begin{proof}
We need to check that $\iota \otimes id_{L(\Gamma)} $ is one-to-one where  $\iota :  A \hookrightarrow B$ is the inclusion of an $\mathbb{F}\Gamma$-submodule.  We will use the model $colim \sigma^k \underline{\> \> \>}\> $ for $\> \underline{\> \> \>} \otimes L(\Gamma)$. If $\alpha \in colim \sigma^k A$ is in $Ker(\iota \otimes id_{L(\Gamma)})$ then $ \alpha = [\alpha_k]$ for some $\alpha_k \in \sigma^k A$ and $k \in \mathbb{N}$. Since $[\alpha_k]=0$ in $colim \sigma^k B$, we have $n\geq k$ with 
$\alpha_k \mapsto 0$ in $ \sigma^n B$. Then   $\alpha_k \in \sigma^k A \subseteq \sigma^k B$ maps to $\alpha_n=0 \in \sigma^n A$ by naturality. Hence $\alpha=[\alpha_n]=0$, i.e., $\underline{\> \> \>} \otimes L(\Gamma)$ is left exact.
\end{proof}\\

\begin{example}
Let $ \Gamma$ be $\quad  \xymatrix{{\bullet} \ar@(ul,ur)}  \quad$ and $\Lambda$ be $\quad  \xymatrix{  { \bullet}
\ar@(ul,ur) } \xymatrix{ \longrightarrow  { \bullet}}$ so $\mathbb{F}\Gamma \cong \mathbb{F}[x] $ and $L(\Gamma)\cong \mathbb{F}[x,x^{-1}] $. Also $C(\Gamma)\cong  \mathbb{F}<x,y> / (1-yx) \cong L(\Lambda)$ and $I \cong M_{\infty}(\mathbb{F})$, the algebra of matrices indexed by $\mathbb{N}$ with only finitely many nonzero entries. The ideal $I$ in $L(\Lambda)$ is generated by the sink in $\Lambda$. The projective resolution 
 $0 \rightarrow   I \rightarrow C(\Gamma) \rightarrow L(\Gamma) \rightarrow 0$ of Fact \ref{resolution} does not split (over $\mathbb{F}\Gamma$) because $\mathbb{F}[x,x^{-1}]$  is not a projective $\mathbb{F}[x]$-module (since $Hom^{\mathbb{F}[x]}( \mathbb{F}[x,x^{-1}], \mathbb{F}[x])=0). $
\end{example}


When $M$ is an $\mathbb{F}\Gamma$-module,  let $\widehat{M} =\{ m\in M \vert   \> \> \exists n\in \mathbb{N} \> \> \forall p \in E_n \>\> \> mp=0 \}$. If $ m \in \widehat{M}$ then $mq \in \widehat{M}$ for all $q \in Path(\Gamma)$, hence $\widehat{M}$ is a submodule of $M$. Also, if $f: M \rightarrow N$ is an $\mathbb{F}\Gamma$-module homomorphism and $m \in \widehat{M}$ then $f(m) \in \widehat{N} $ 
so  $\> \widehat{\> \> \> \> \>  }\> $ is an endofunctor on $\mathfrak{M}_{\mathbb{F}\Gamma}$ with $\widehat{f}:= f \vert_{\widehat{M}}$.
The inclusions $\widehat{M} \hookrightarrow M$ define a natural transformation from  $\>  \widehat{ \> \> \> \> \> } \> $ to the identity functor. If $A$ is a submodule of $M$ then $\widehat{A} =A \cap \widehat{M}$, therefore  $\> \widehat{\> \> \>  \> \> }\> $ is left exact.

 \begin{lemma} If $M$ is an $L(\Gamma)$-module then the following are equivalent:\\
(i)  $\widehat{M}=0$; \\
(ii) $(\cdot e)_{se=v} : Mv \longrightarrow  \oplus_{se=v} Mte$ for all nonsink $v\in V$ is one-to-one;\\
(iii) $(\cdot p)_{p\in E_n^v} : Mv \longrightarrow  \oplus_{sp=v} Mtp$ for all $v \in V$ and for all $n \in \mathbb{N}$ is one-to-one.
\end{lemma}
\begin{proof}
$(i) \Rightarrow (ii) : \quad Ker( (\cdot e)_{se=v}) \subseteq \widehat{M}$. \\
$(ii) \Rightarrow (iii) :$ Since $E_1$ is the union of $E$ and all sinks in $\Gamma$,  this follows from $(E_1)^n=E_n$ and the fact that composition of one-to-one functions is one-to-one.\\
$(iii) \Rightarrow (i) :$ This is immediate from the definition of $\widehat{M}$.
 \end{proof}

\begin{theorem} \label{kernel}
$\widehat{M}$ is the kernel of $M \longrightarrow colim \sigma^kM$, equivalently the kernel of the composition $M \cong M \otimes_{\mathbb{F}\Gamma} \mathbb{F}\Gamma \stackrel{id_M \otimes \iota}{\longrightarrow} M \otimes_{\mathbb{F}\Gamma} L(\Gamma)$.
\end{theorem}
\begin{proof}
First we will prove this assuming $\widehat{M}=0$. The homomorphism $M \longrightarrow \sigma^kM$ is one-to-one if $\widehat{M}=0$. Then $colim \sigma^kM$ can be identified with $\cup \sigma^kM$. Hence $M \longrightarrow colim \sigma^kM$ has kernel $0=\widehat{M}$. Note that $\widehat{ M/ \widehat{M}}=0$.  
Now, consider the commutative square 
\[
\left.
\begin{array}{ccc}
  M & \longrightarrow   &M/ \widehat{M}  \\
 \downarrow  &   &\downarrow   \\
  M \otimes L(\Gamma) & \longrightarrow   &   M/ \widehat{M}\otimes L(\Gamma) 
  \end{array}
\right.
\]
If $m \in Ker(M \longrightarrow M \otimes L(\Gamma))$ then $m \in \widehat{M}$ since  $M/ \widehat{M}  \longrightarrow M/ \widehat{M}\otimes L(\Gamma)$ is one-to-one as shown above.
 \end{proof}

\begin{theorem} \label{hat} Let $f:A \longrightarrow B$ be a homomorphism of $\mathbb{F}\Gamma$-modules. Then $f \otimes_{\mathbb{F}\Gamma} id_{L(\Gamma)}=0$ if and only if $f(A) \subseteq \widehat{B}$.
 \end{theorem}
\begin{proof}
If $f(A) \subseteq \widehat{B}$ then for all $v \in V$ and for all $a \in Av$ there is a $k \in \mathbb{N}$ such that $f(a)p=0$ for all $p \in E_k^v$. Now $f(a) \otimes v= f(a)\otimes \sum pp^*=\sum f(ap)\otimes p^*=0$ where the sum is over all $p$ in $E_k^v$. 
Also $f(a) \otimes \lambda=f(av)\otimes \lambda= (f(a)\otimes v) \lambda =0$ for all $\lambda$ in  $L(\Gamma)$. Since $A =\oplus_{v \in V} Av$, we get that $f \otimes id_{L(\Gamma)}=0$.\\

Conversely,  if $f \otimes_{\mathbb{F}\Gamma} id_{L(\Gamma)}=0$ then in the commutative square 
\[
\left.
\begin{array}{ccc}
  A&\stackrel{f}{\longrightarrow}    & B  \\
  \downarrow &   & \downarrow  \\
  A\otimes L(\Gamma) & \stackrel{0}{\longrightarrow}   &   B \otimes L(\Gamma)
\end{array}
\right.
\]
$f(A) \subseteq Ker (B \longrightarrow B\otimes L(\Gamma))=\widehat{B}$ by Theorem \ref{kernel}.
\end{proof}
\begin{corollary} \label{sapka}
$M\otimes L(\Gamma)=0$ for an $\mathbb{F}\Gamma$-module $M$ if and only if $M =\widehat{M}$.
\end{corollary}
\begin{proof}
If $M =\widehat{M}$ then the inclusion $\widehat{M} \hookrightarrow M $ is $id_M$ and $id_M(\widehat{M})\subseteq \widehat{M}$ hence $id_{M \otimes L(\Gamma)}=id_M \otimes id_{L(\Gamma)}=0$ by Theorem \ref{hat}. Thus $M\otimes L(\Gamma)=0$.\\

Conversely, if $colim \sigma^kM \cong M\otimes L(\Gamma)=0$ then $[m]=0$ in $colim \sigma^kM$ for all $m \in M$. Hence $m \in Ker (M\longrightarrow \sigma^n M)$ for some $n \in \mathbb{N}$. Thus $m \in \widehat{M}$ and $M=\widehat{M}$.
\end{proof}

\begin{corollary}
Let $M$ be an $\mathbb{F}\Gamma$-module and let $\iota :\mathbb{F}\Gamma \longrightarrow L(\Gamma)$ be the standard algebra homomorphism. 
If  for all nonsink $v \in V $  $$\xymatrix{ Mv \ar[r]^{(\cdot e)} & \bigoplus\limits_{se=v} Mte }$$  
 is one-to-one then $\xymatrix{ M \cong M\otimes \mathbb{F}\Gamma \ar[r]^{\> \> id_M \otimes \iota} &M\otimes L(\Gamma)}$ is one-to-one. 
\end{corollary}
\begin{proof}
Since composition of one-to-one functions is again one-to-one, the hypothesis implies that $\xymatrix{ Mv \ar[r]^{(\cdot p)} & \bigoplus\limits_{p \in E_n^v} Mtp }$ is one-to-one for all $n \in \mathbb{N}$. Hence $\widehat{M}=0$ and $M \longrightarrow M\otimes L(\Gamma)$ is one-to-one by Theorem \ref{kernel}.
  \end{proof}


\begin{theorem} \label{Serre}
The full subcategory $\widehat{\mathfrak{M}}_{\Gamma}$ of quiver representations  $\mathfrak{M}_{\mathbb{F}\Gamma}$ with objects $M=\widehat{M}$ is a Serre subcategory. The quotient category $\mathfrak{M}_{\mathbb{F}\Gamma}/ \widehat{\mathfrak{M}}_{\Gamma} $ is equivalent to $\mathfrak{M}_{L(\Gamma)}$.
\end{theorem}
\begin{proof}
If $A$ is a $\mathbb{F}\Gamma$-submodule of $B$ then $\widehat{A}= \widehat{B} \cap A$ and  if $f$ is a $\mathbb{F}\Gamma$-module homomorphism then $f(\widehat{B}) \subseteq \widehat{f(B)}$. It follows that $\widehat{\mathfrak{M}}_{\Gamma}$ is closed under subquotients.
If $A$ is a submodule of $B$ with $\widehat{A}=A$ and $\widehat{B/A}= B/A$ then for $b \in Bv$ there is $n$ such that $bp \in A$  for all $p \in E_n^v$ since $\widehat{B/A}= B/A$. Also for each $bp$ there is $n_p$ such that $bpq=0$ for all $q \in E_{n_p}^{tp}$. Let  $k=max\{{n_p}\> \vert\>  p \in E_n^v \}$. Now $bp=0$ for all $p \in E_{n+k}^v$ hence $\widehat{B}=B$ and $\widehat{\mathfrak{M}}_{\Gamma}$ is a Serre subcategory since it is a full subcategory closed under sub quotients and extensions. \\

Recall that the quotient category $\mathfrak{M}_{\mathbb{F}\Gamma}/ \widehat{\mathfrak{M}}_{\Gamma} $ has the same objects as $\mathfrak{M}_{\mathbb{F}\Gamma}$. The morphisms are $colim \>  Hom(A' , B/B')$ where the direct limit is over ordered pairs $(A', B')$ such that  $A' \leq A$ with $\widehat{A/A'}=A/A'$ and $B' \leq B$ with $\widehat{B'}=B'$. The functor from $\mathfrak{M}_{L(\Gamma)}$ to $\mathfrak{M}_{\mathbb{F}\Gamma}/ \widehat{\mathfrak{M}}_{\Gamma} $ is the forgetful functor from  $\mathfrak{M}_{L(\Gamma)}$  to  $\mathfrak{M}_{\mathbb{F}\Gamma}$ composed with the quotient functor to $\mathfrak{M}_{\mathbb{F}\Gamma}/ \widehat{\mathfrak{M}}_{\Gamma} $. \\

The functor from $\mathfrak{M}_{\mathbb{F}\Gamma}/ \widehat{\mathfrak{M}}_{\Gamma} $ to $\mathfrak{M}_{L(\Gamma)}$ is induced by $ \underline{\> \>\>} \otimes_{\mathbb{F}\Gamma} L(\Gamma)$. The inclusion $A' \hookrightarrow A$ and the quotient $B \longrightarrow B/B'$ give isomorphisms $A\otimes_{\mathbb{F}\Gamma} L(\Gamma) \cong A' \otimes_{\mathbb{F}\Gamma} L(\Gamma)$ and 
 $B/B' \otimes_{\mathbb{F}\Gamma} L(\Gamma) \cong B\otimes_{\mathbb{F}\Gamma} L(\Gamma)$ when $\widehat{A/A'}=A/A'$ and $\widehat{B'}=B'$ by Lemma \ref{flat} and Corollary \ref{sapka}. The image of a morphism $[A' \stackrel{f}{\longrightarrow} B/B']$
 is defined via the isomorphisms $A\otimes_{\mathbb{F}\Gamma} L(\Gamma) \cong A' \otimes_{\mathbb{F}\Gamma} L(\Gamma)$ and 
 $B/B' \otimes_{\mathbb{F}\Gamma} L(\Gamma) \cong B\otimes_{\mathbb{F}\Gamma} L(\Gamma)$. Well-definedness follows from  the naturality 
 of these isomorphisms. \\

 The composition of the two functors above from $\mathfrak{M}_{L(\Gamma)}$ to itself is naturally  equivalent to the identity functor on $\mathfrak{M}_{L(\Gamma)}$ by Theorem \ref{forgetful}. The natural transformation from the identity functor on $\mathfrak{M}_{\mathbb{F}\Gamma}/ \widehat{\mathfrak{M}}_{\Gamma} $ to the other composition is given by (the equivalence class of) :
 $M \cong M\otimes_{\mathbb{F}\Gamma} \mathbb{F}\Gamma \longrightarrow M\otimes_{\mathbb{F}\Gamma} L(\Gamma)$, more explicitly 
$m=\sum m_iv_i \mapsto \sum m_i \otimes v_i$ for all $m \in M$.
This defines an isomorphism in $\mathfrak{M}_{\mathbb{F}\Gamma}/ \widehat{\mathfrak{M}}_{\Gamma} $ because its kernel is $\widehat{M}$ (by Theorem \ref{kernel}) and its cokernel is also an object of $\widehat{\mathfrak{M}}_{\Gamma}$: Given $\sum_{i=1}^k m_i^*q_i^*$ in $M\otimes L(\Gamma)$ if $n=max \{l(q_i)\}_{i=1}^k$ then $(\sum m_i^*q_i^*)p$ is in the image of $M$ for all $p\in E_n$. Since both compositions are naturally equivalent to the identity functors, the categories $\mathfrak{M}_{\mathbb{F}\Gamma}/ \widehat{\mathfrak{M}}_{\Gamma} $ and  $\mathfrak{M}_{L(\Gamma)}$ are equivalent. 
\end{proof} \\

\textbf{Acknowledgements}\\

Ayten Ko\c{c} was supported by TUBITAK (The Scientific and Technological Research Council of Turkey)  BIDEB 2219-International Post-Doctoral Research Fellowship during her visit to the University of Oklahoma where some of this research was done. She would like to thank her colleagues at the host institution for their hospitality. Both authors were also partially supported by TUBITAK ARDEB 1001 grant 115F511.
The authors would also like to thank an anonymous referee for a careful reading of this manuscript, several suggestions improving the exposition
and making us aware of the reference \cite{allw18} which lead to a strengthening of Corollary 6.6.




\noindent
$^*$ Department of Mathematics, \\ Gebze Technical University,  41400 Gebze/Kocaeli, TURK\.{I}YE\\ 
E-mail: aytenkoc@gtu.edu.tr\\

\noindent
$^{**}$ Department of Mathematics, \\ University of Oklahoma, Norman, OK, USA \\
E-mail: mozaydin@ou.edu

\begin{thebibliography}{27}

\bibitem{abr15} G. Abrams, Leavitt Path Algebras: The First Decade, Bull. Math. Sci. (2015) 5:59-120.

\bibitem{aas} G. Abrams, P. Ara, M. Siles Molina, Leavitt Path Algebras, Lecture Notes in Mathematics Vol. 2191, Springer Verlag, 2017.

\bibitem{aa05} G. Abrams, G. Aranda Pino, The Leavitt Path Algebra of a Graph, Journal of Algebra 293 (2) (2005) 319–334.


\bibitem{aaps} G. Abrams, G. Aranda Pino, F. Perera, M. Siles Molina, Chain Conditions for Leavitt Path Algebras, Forum Mathematicum 22 (1) (2010) 95-114.


\bibitem{aas07} G. Abrams, G. Aranda Pino, M. Siles Molina, Finite-dimensional Leavitt Path Algebras, Journal of Pure and Applied Algebra 209 (2007) 753–762.

\bibitem{aas08} G. Abrams, G. Aranda Pino, M. Siles Molina, Locally Finite Leavitt Path Algebras, Israel Journal of Mathematics 165 (2008) 329 - 348.

\bibitem{ak} G. Abrams, M. Kanuni, Cohn Path Algebras Have IBN, Communications in Algebra 44 (2016) 371-380.
\bibitem{amp17} G. Abrams, F. Mantese, A. Tonolo, Leavitt Path Algebras are Bezout, arXiv:1605.08317v1 [math.RA], https://arxiv.org/abs/1605.08317.
\bibitem{anp} G. Abrams, T. G. Nam, N. T. Phuc, Leavitt Path Algebras Having Unbounded Generating Number, Journal of Pure and Applied Algebra 211 (2017) 1322-1343.

\bibitem{ar10} G. Abrams, K. M. Rangaswamy, Regularity Conditions for Arbitrary Leavitt Path Algebras, Algebras and Representation Theory 13 (3) (2010) 319-334.


 
  \bibitem{aajz12} A. Alahmedi, H. Alsulami, S. Jain, E. Zelmanov, Leavitt Path Algebras of Finite Gelfand-Kirillov Diemension J. Algebra Appl. 11 (6) (2012) 1250225, 1-6.

\bibitem{aajz13} A. Alahmedi, H. Alsulami, S. Jain, E. Zelmanov, Structure of Leavitt Path Algebras of  Polynomial Growth, Proceedings of the National Academy of Sciences USA 110 (2013) 15222-15224.
	
\bibitem{ab10} P. Ara, M. Brustenga, Module Theory over Leavitt Path Algebras and K-theory, J. Pure Appl. Algebra 214 (2010) 1131-1151.


\bibitem{ag12} P. Ara, K.R. Goodearl, Leavitt Path Algebras of Separated Graphs, J. Reine Angew. Math. 669 (2012) 165-224. 

\bibitem{allw18} P. Ara, K. Li, F. Lledo, J. Wu, Amenability of Coarse Spaces and $\mathbb{K}$-algebras, Bull. Math. Sci. (2018) 8:257-306.


\bibitem{amp07} P. Ara, M.A. Moreno, E. Pardo, Nonstable  $\mathbb{K}$-theory for Graph Algebras, Algebras and Representation Theory 10 (2007) 157--178.

\bibitem{ar14} P. Ara, K.M. Rangaswamy,  \textit{Finitely Presented Simple Modules over Leavitt Path Algebras}, J. Algebra 417 (2014) 333-352.

%
%

\bibitem{ber74} G.M. Bergman, Coproducts and Some Universal Ring Constructions, Trans. Amer. Math. Soc. 200 (1974) 33-88.


\bibitem{che15} X.W. Chen, Irreducible Representations of Leavitt Path Algebras, Forum Mathematicum 27 (2015) 549-574.

\bibitem{dw05} H. Derksen, J. Weyman, An Introduction to Quiver Representations, GSM, Vol. 184, AMS,  2017.

\bibitem{goo09}K. R. Goodearl,  Leavitt Path Algebras and Direct Limits in Rings and Representations, Contemporary Math. Series 480, Amer. Math. Soc., Providence, RI (2009) 165-188.






\bibitem{jac50} N. Jacobson, Some Remarks on One-sided Inverses, Proc AMS 1 (1950) 352-355.






\bibitem{ko} M. Kanuni, M. Özaydın, Leavitt Path Algebras  and the IBN Property, arXiv:1606.07998v1 [math.RA], to appear in Journal of Algebra and its Applications.

\bibitem{koc1} A. Ko\c{c}, M. Özaydın, Finite Dimensional Representations of Leavitt Path Algebras, Forum Mathematicum, 30 (4), 915-928, (2018).

\bibitem{kegk14} A. Ko\c{c}, S. Esin, \.{I}. G\"{u}lo\u{g}lu, M. Kanuni, A Combinatorial Discussion on Finite Dimensional Leavitt Path Algebras, Hacettepe Journal of Mathematics and Statistics 43 (6) (2014) 943-951.



\bibitem{lea62} W. G. Leavitt, The Module Type of a Ring, Transactions AMS, 103 (1962) 113-130.


\bibitem{nrs04} A. Neeman, A. Ranicki, A. Schofield, Representations of Algebras as Universal Localizations, Math. Proc. 
Camb. Phil. Soc. 136 (2004) 105-117.




\end{thebibliography}
\end{document}